\documentclass[a4paper,11pt]{article}

\usepackage[T1]{fontenc} 
\usepackage[utf8]{inputenc} 
\usepackage[english]{babel}
\usepackage{geometry}

\usepackage{dsfont}
\usepackage{color}

\usepackage{amsmath,amsfonts,amssymb,amsthm,mathrsfs,dsfont,mathtools}
\usepackage[hidelinks]{hyperref}
\usepackage{hyperref}
\hypersetup{
    colorlinks=true,
	linkcolor=blue,
	citecolor=blue,
	urlcolor=blue
}

\usepackage[capitalize,nameinlink]{cleveref}

\usepackage{paralist}
\usepackage[shortlabels]{enumitem}
\usepackage{textcase} 
\usepackage{filemod} 
\usepackage[longnamesfirst]{natbib}
\usepackage{titlesec} 
\usepackage{etoolbox} 

\usepackage[final]{showlabels}

\makeatletter

\def\cite{\citet}

\numberwithin{equation}{section}
\allowdisplaybreaks

\def\@noindentfalse{\global\let\if@noindent\iffalse}
\def\@noindenttrue {\global\let\if@noindent\iftrue}
\def\@aftertheorem{%
  \@noindenttrue
  \everypar{%
    \if@noindent%
      \@noindentfalse\clubpenalty\@M\setbox\z@\lastbox%
    \else%
      \clubpenalty \@clubpenalty\everypar{}%
    \fi}}

\theoremstyle{plain}
\newtheorem{theorem}{Theorem}[section]
\AfterEndEnvironment{theorem}{\@aftertheorem}

\AfterEndEnvironment{definition}{\@aftertheorem}
\newtheorem{lemma}[theorem]{Lemma}
\AfterEndEnvironment{lemma}{\@aftertheorem}

\AfterEndEnvironment{corollary}{\@aftertheorem}
\newtheorem{proposition}[theorem]{Proposition}
\AfterEndEnvironment{proposition}{\@aftertheorem}

\theoremstyle{definition}
\newtheorem{remark}[theorem]{Remark}
\AfterEndEnvironment{remark}{\@aftertheorem}

\AfterEndEnvironment{example}{\@aftertheorem}    
\AfterEndEnvironment{proof}{\@aftertheorem}

\titleformat{name=\section}
  {\center\small\sc}{\thesection\kern1em}{0pt}{\MakeTextUppercase}
\titleformat{name=\section, numberless}
  {\center\small\sc}{}{0pt}{\MakeTextUppercase}

\titleformat{name=\subsection}
  {\bf\mathversion{bold}}{\thesubsection\kern1em}{0pt}{}
\titleformat{name=\subsection, numberless}
  {\bf\mathversion{bold}}{}{0pt}{}

\titleformat{name=\subsubsection}
  {\normalsize\itshape}{\thesubsubsection\kern1em}{0pt}{}
\titleformat{name=\subsubsection, numberless}
  {\normalsize\itshape}{}{0pt}{}

\def\note#1{\par\smallskip%
\noindent\kern-0.01\hsize%
\setlength\fboxrule{0pt}\fbox{\setlength\fboxrule{0.5pt}\fbox{%
\llap{$\boldsymbol\Longrightarrow$ }%
\vtop{\hsize=0.98\hsize\parindent=0cm\small\rm #1}%
\rlap{$\enskip\,\boldsymbol\Longleftarrow$}
}%
}%
}

\usepackage{delimset}
\def\given{\mskip 0.5mu plus 0.25mu\vert\mskip 0.5mu plus 0.15mu}
\newcounter{bracketlevel}%
\def\@bracketfactory#1#2#3#4#5#6{%
\expandafter\def\csname#1\endcsname##1{%
\global\advance\c@bracketlevel 1\relax%
\global\expandafter\let\csname @middummy\alph{bracketlevel}\endcsname\given%
\global\def\given{\mskip#5\csname#4\endcsname\vert\mskip#6}\csname#4l\endcsname#2##1\csname#4r\endcsname#3%
\global\expandafter\let\expandafter\given\csname @middummy\alph{bracketlevel}\endcsname%
\global\advance\c@bracketlevel -1\relax%
}%
}
\def\bracketfactory#1#2#3{%
\@bracketfactory{#1}{#2}{#3}{relax}{0.5mu plus 0.25mu}{0.5mu plus 0.15mu}
\@bracketfactory{b#1}{#2}{#3}{big}{1mu plus 0.25mu minus 0.25mu}{0.6mu plus 0.15mu minus 0.15mu}
\@bracketfactory{bb#1}{#2}{#3}{Big}{2.4mu plus 0.8mu minus 0.8mu}{1.8mu plus 0.6mu minus 0.6mu}
\@bracketfactory{bbb#1}{#2}{#3}{bigg}{3.2mu plus 1mu minus 1mu}{2.4mu plus 0.75mu minus 0.75mu}
\@bracketfactory{bbbb#1}{#2}{#3}{Bigg}{4mu plus 1mu minus 1mu}{3mu plus 0.75mu minus 0.75mu}
}
\bracketfactory{clc}{\lbrace}{\rbrace}
\bracketfactory{clr}{(}{)}
\bracketfactory{cls}{[}{]}
\bracketfactory{abs}{\lvert}{\rvert}
\bracketfactory{norm}{\lVert}{\rVert}
\bracketfactory{floor}{\lfloor}{\rfloor}
\bracketfactory{ceil}{\lceil}{\rceil}
\bracketfactory{angle}{\langle}{\rangle}

\let\original@left\left
\let\original@right\right
\renewcommand{\left}{\mathopen{}\mathclose\bgroup\original@left}
\renewcommand{\right}{\aftergroup\egroup\original@right}

\newcounter{ctr}\loop\stepcounter{ctr}\edef\X{\@Alph\c@ctr}%
	\expandafter\edef\csname s\X\endcsname{\noexpand\mathscr{\X}}
	\expandafter\edef\csname c\X\endcsname{\noexpand\mathcal{\X}}
	\expandafter\edef\csname b\X\endcsname{\noexpand\boldsymbol{\X}}
	\expandafter\edef\csname I\X\endcsname{\noexpand\mathbb{\X}}
\ifnum\thectr<26\repeat

\let\@IE\IE\let\IE\undefined
\newcommand{\IE}{\mathop{{}\@IE}\mathopen{}}
\newcommand{\E}{\mathop{{}\@IE}\mathopen{}}
\let\@IP\IP\let\IP\undefined
\newcommand{\IP}{\mathop{{}\@IP}\mathopen{}}
\renewcommand{\P}{\mathop{{}\@IP}\mathopen{}}
\newcommand{\Var}{\mathop{\mathrm{Var}}}
\newcommand{\Cov}{\mathop{\mathrm{Cov}}}

\def\R{\IR}
\def\^#1{\relax\ifmmode {\mathaccent"705E #1} \else {\accent94 #1} \fi}
\def\~#1{\relax\ifmmode {\mathaccent"707E #1} \else {\accent"7E #1} \fi}
\def\*#1{\relax#1^\ast}
\edef\-#1{\relax\noexpand\ifmmode {\noexpand\bar{#1}} \noexpand\else \-#1\noexpand\fi}
\def\>#1{\vec{#1}}
\def\.#1{\dot{#1}}

\def\atop{\@@atop}

\renewcommand{\leq}{\leqslant}
\renewcommand{\geq}{\geqslant}
\renewcommand{\phi}{\varphi}

\newcommand{\I}{\mathop{{}\mathrm{I}}\mathopen{}}

\newcommand\indep{\protect\mathpalette{\protect\@indep}{\perp}}
\def\@indep#1#2{\mathrel{\rlap{$#1#2$}\mkern2mu{#1#2}}}

\def\parsetime#1#2#3#4#5#6{#1#2:#3#4}
\def\parsedate#1:20#2#3#4#5#6#7#8+#9\empty{20#2#3-#4#5-#6#7 \parsetime #8}
\def\moddate{\expandafter\parsedate\pdffilemoddate{\jobname.tex}\empty}

\crefname{equation}{}{}
\crefformat{equation}{\textrm{(#2#1#3)}}
\crefrangeformat{equation}{\textup{(#3#1#4)--(#5#2#6)}}

\crefname{page}{p.}{pp.}

\newenvironment{equ}
{\begin{equation} \begin{aligned}}
{\end{aligned} \end{equation}}
\AfterEndEnvironment{equ}{\noindent\ignorespaces}

\newlist{condition}{enumerate}{10}
\setlist[condition]{label*=({A}\arabic*)}
\crefname{conditioni}{condition}{conditions}
\Crefname{conditioni}{Condition}{Conditions}

\makeatother

\begin{document}

\title{\sc\bf\large\MakeUppercase{Berry--Esseen bounds for self-normalized sums of local dependent random variables}}
\author{\sc Zhuo-Song Zhang}

\date{\itshape National University of Singapore}

\maketitle


\maketitle
\begin{abstract}
	In this paper, we prove a Berry--Esseen bound with optimal order for self-normalized sums of local dependent random variables under some mild dependence conditions. The proof is based on Stein's method and a randomized concentration inequality. 
	As applications, we obtain optimal Berry--Esseen bounds for $m$-dependent random variables and graph dependency. 
\end{abstract}



\section{Introduction}
\label{sec:1}
Let $X_1, X_2, \dots $ be a sequence of independent random variables, and let 
\begin{align*}
	S_n = \sum_{i = 1}^n X_i, \quad V_n^2 = \sum_{i = 1}^n X_i^2, \quad \text{ and } \quad  \hat\sigma_n^2 = \frac{n - 1}{n} \sum_{i = 1}^n (X_i - \bar X)^2, 
\end{align*}
where $\bar X = S_n/n$. We say $S_n/V_n$ is a \emph{self-normalized sum}. We note that $S_n/\hat\sigma_n$ is the well-known Student's $t$-statistic, which is one of the most important tools in statistical testing
when the standard deviation of the underlying distribution is unknown. 
Based on the fact that 
\begin{align*}
	\P(S_n/\hat\sigma_n \geq x) 
	= \P(S_n/V_n \geq x[ n / (n + x^2 - 1) ]^{1/2}), \quad x \geq 0,
\end{align*}
it is common to study the self-normalized sum $S_n/V_n$. 
One of the most important advantages of self-normalized sums is that the range of Gaussian approximation can be much wider than their corresponding non-self-normalized sums under the same polynomial moment conditions. 
Berry--Esseen bounds of Gaussian approximation for self-normalized sums of independent random variables have been well-studied in the literature.  For example,  
Berry--Esseen theorem for Student's $t$-statistics were proved in \cite{Ben96} and \cite{Ben96a}, and  an exponential nonuniform Berry--Esseen bound was obtained in \cite{Jin99}. Differing from the method in \cite{Ben96} and \cite{Ben96a}, Stein's method can be used to prove
Berry--Esseen bounds with
explicit constants (c.f. \cite{shao_explicit_2005}).

If $(X_1, X_2, \dots)$ is a family of dependent random variables, asymptotic theories on self-normalized sums  have also been studied in the literature. 
For example, in \cite{de04}, the authors studied self-normalized processes,  and in \cite{Ber08a}, 
self-normalized sums for martingales were considered.
Recently,  a Cram\'er-type moderate deviation under some weak dependence assumptions was proved in \cite{Che16}. For more details, we refer to \cite{Sha13a} for a
survey.  

Local dependence is also a commonly used dependence structure in applications. A family of local dependent random variables means that certain subset of the random variables are independent of those outside their respective neighborhoods. 
There are several forms of local dependence assumptions in the literature. For example,  \emph{decomposable random variables} were considered in \cite{Bar89}, where a $L_1$ bound was also obtained; \emph{dependency neighborhoods} was introduced in \cite{Ri96M} to study error bounds for multivariate normal
approximation. Uniform  and nonuniform Berry--Esseen bounds for a general local dependence structure were established in \cite{Che04b}. Recently, Fang \cite{Fan19b} proved an error bound of Wasserstein-2 distance for a generalized local dependence structure.

In this paper, our main purpose is to prove Berry--Esseen bounds for self-normalized sums of local dependent random variables. 
In Theorem 2.1 (see Section 2), under conditions (LD1) and (LD2), we provide a Berry--Esseen bound for self-normalized sums of local dependent random variables. 
Compared to the results for non-self-normalized sums in Theorem 2.2 in \cite{Che04b}, under the same conditions (LD1) and (LD2), we do not require the fourth moment assumption to obtain the optimal convergence rate, which in turn shows robustness of self-normalized statistics. 
As applications, we obtain Berry--Esseen bounds for studentized statistics for $m$-dependent random variables and graph dependency. 

The proof of our main result is based on Stein's method and concentration inequality approach. 
The technique of concentration inequality approach has been applied to obtain sharp Berry--Esseen bounds for univariate and multivariate normal approximations in the literature, and we refer to
\cite{chen_steins_1998,chen_non-uniform_2001,Che04b,chen_normal_2007,shao_explicit_2005,shao_cramer_2016} and \cite{shao_berry--esseen_2021}.  In this paper, we develop new randomized concentration inequalities for local dependent random variables under some mild dependence conditions. The ideas are based on \cite{Che04b} and \cite{shao_explicit_2005}. 

The rest of this paper is organized as follows. We give our main results in Section 2. Applications are given in Section 3. Some useful preliminary lemmas and a new randomized concentration inequality are proved in Section 4. We give the proof of our main result in Section
5.

\section{Main results}

Let $\{X_i\}_{i \in \mathcal{J}}$ be a field of real-valued random variables satisfying that $\E  X_i  = 0$ for all $i \in \mathcal{J}$. 
We introduce the following  local dependence conditions. 
\begin{enumerate}[({LD}1)]
	\item For any $i \in \mathcal{J}$, there exists $A_i \subset \mathcal{J}$ such that $X_i$ is independent of $\{X_j : j \not\in A_i\}$; 
	\item For any $i \in \mathcal{J}$, there exists $B_i \subset \mathcal{J}$ such that $B_i \supset A_i$ and $\{ X_j : j \in A_i \}$ is independent of $\{X_j : j \not\in B_i\}$.
\end{enumerate}
These dependence assumptions have been commonly used in the literature. We remark that we do not assume any structures on the index set. 
Put 
\begin{equ}
	S & =  \sum_{i \in \mathcal{J}} X_i , &
	V & =  \sqrt{ \Bigl( \sum_{i \in \mathcal{J}} (X_i Y_i - \bar X \bar Y) \Bigr)_{+} },  &
	W & =  S / V, 
    \label{eq-SVW}
\end{equ}
where $Y_i = \sum_{j \in A_i} X_i$, $\bar X  = \sum_{i \in \mathcal{J}} X_i / \lvert \mathcal{J} \rvert$, $\bar Y = \sum_{j \in \mathcal{J}} Y_j / \lvert \mathcal{J} \rvert$ and $(x)_{+} = \max(x, 0)$ is the positive part of $x$. 
Let $\sigma = (\Var(S))^{1/2}$. 
Here and in the sequel, we denote by $|A|$ the cardinality
of $A$ for any $A \subset \mathcal{J}$. 
We remark that if $\E  X_i $ is not necessarily 0, then one can simply replace $X_i$ by $X_i - \E  X_i $ in \cref{eq-SVW}. 

Moreover, let $\kappa$ be any number such that $\kappa\geq \max_{i \in \mathcal{J}} \{\lvert \{ j: B_i \cap A_j \neq \emptyset \}\rvert , \lvert\{ j: i \in B_j \} \rvert$ and let 
\begin{align*}
	\beta_0 & =  \sum_{i \in \mathcal{J}} \P( \lvert X_i \rvert > \sigma / \kappa ), \\
	\beta_2 & =  \frac{1}{\sigma^2} \sum_{i \in \mathcal{J}} \E \{ \lvert X_i\rvert^2 \I ( \lvert X_i \rvert > \sigma/\kappa ) \}, \\
	\beta_3 & =  \frac{1}{\sigma^3} \sum_{i \in \mathcal{J}} \E \{ \lvert X_i \rvert^3 \I ( \lvert X_i \rvert \leq \sigma/\kappa ) \}, \\
	\theta  & = \frac{1}{\sigma^2} \sum_{i \in \mathcal{J}} \sum_{j \in A_i} \E \lvert X_i X_j  \I( |X_i|  \leq \sigma / \kappa, \lvert X_j \rvert \leq \sigma / \kappa ) \rvert.
\end{align*}

We have the following theorem. 

\begin{theorem}
\label{thm1}
	Under (LD1) and (LD2). 
We have
	\begin{equ}
		\MoveEqLeft \sup_{z \in \R} \bigl\lvert \P(W \leq z) - \Phi(z) \bigr\rvert\\
		& \leq C \{  (1 + \theta) \kappa^2 \beta_3+ \kappa \beta_2 + \beta_0\} + C \kappa^{1/2} (\theta + 1) |\mathcal{J}|^{-1/2} , 
        \label{eq-thm1}
	\end{equ}
	where $C > 0$ is an absolute constant and $\Phi$ is the standard normal distribution function.
\end{theorem}
\begin{remark}
	We first make some remarks on $\beta_j$'s. If $\max_{i}\lvert X_i \rvert\leq \sigma / \tau$ almost surely, then $\beta_0 = 0$. Otherwise, by Chebyshev's inequality, we have 
	\begin{align*}
		\beta_0 \leq \kappa^2 \beta_2. 
	\end{align*}
	If $\E \lvert X_i \rvert^3 < \infty$, then 
	\begin{align*}
		\kappa\beta_2 + \kappa^2\beta_3 \leq \frac{\kappa^2}{\sigma^3} \sum_{i \in \mathcal{J}} \E \lvert X_i \rvert^3. 
	\end{align*}
	Moreover, by H\"older's inequality, we have 
	\begin{equ}
		\theta 
		& \leq \biggl( \E \biggl\{ \frac{1}{\sigma^2} \sum_{i \in \mathcal{J}} \sum_{j \in A_i} \lvert X_i X_j \rvert \I( \lvert X_i \rvert \leq \sigma / \kappa , \lvert X_j \rvert \leq \sigma / \kappa) \biggr\}^{3/2} \biggr)^{2/3}\\
		& \leq \kappa_1 \lvert \mathcal{J} \rvert^{1/3} \beta_3^{2/3} \leq \kappa \lvert \mathcal{J} \rvert^{1/3} \beta_3^{2/3}, 
        \label{eq-theta-bound}
	\end{equ}
	where $\kappa_1 = \max_{i \in \mathcal{J}}  \lvert A_i \rvert $. If $\beta_3$ is of order $O(\lvert \mathcal{J}\rvert^{-1/2})$, then $\theta$ is of order $O(1)$.
\end{remark}

\begin{remark}
	Chen and Shao \cite{Che04b} proved a Berry--Esseen bound of the same order as in \cref{thm1} for non-self-normalized sums. Specifically, assume further that  for any $i \in \mathcal{J}$, there exists $C_i \subset \mathcal{J}$ such that $C_i \supset B_i$ and  $\{ X_j : j \in B_i \}$ is independent of $\{X_j : j \not\in C_i\}$,  and 
it follows that (see Theorem 2.4 of \cite{Che04b})
\begin{align*}
	\sup_{z \in \R}\lvert \P(S/\sigma \leq z) - \Phi(z) \rvert \leq 75 (\kappa')^{p - 1} \sum_{i \in \mathcal{J}} \E \lvert X_i/\sigma \rvert^{p},
\end{align*}
where $\kappa' \geq \max ( \lvert \{j: B_j \cap C_i \neq \emptyset\} \rvert, \lvert \{j : i \in C_j\} \rvert )$.	Under the same dependence condition as in Theorem 1.1,  Chen and Shao \cite{Che04b} proved that for $2  < p \leq 4$, 
	\begin{align*}
		\sup_{z \in \R} \bigl\lvert \P( S/ \sigma \leq z) - \Phi(z) \bigr\rvert & \leq C (1 + \kappa) \sum_{i \in \mathcal{J}} \bigl( \E |X_i / \sigma|^{3 \wedge p} + \E |Y_i/\sigma|^{3 \wedge p} \bigr) \\
		& \quad + C \kappa^{1/2}\biggl(\sum_{i \in \mathcal{J}} \biggl( \E \lvert X_i/\sigma \rvert^p + \E \lvert Y_i/\sigma \rvert^p \biggr) \biggr)^{1/2}, 
	\end{align*}
	which is of the best possible order $O(n^{-1/2})$ when $p = 4$ and $\kappa = O(1)$. For the self-normalized sum $S / V$, the best possible order can be obtained under a third moment condition. 
\end{remark}
\begin{remark}
	Specially, if $\{ X_i, i \in \mathcal{J} \}$ is a field of independent random variables, then it is not hard to see that $\kappa = \theta = 1$. Then, by \cref{eq-theta-bound}, $\lvert J \rvert^{-1/2} \leq \beta_3$. Hence, the right hand side of \cref{eq-thm1} reduces to $C (\beta_0 + \beta_2+ \beta_3 )$, which is as same as the
	results in 
	\cite{Ben96a}.  
\end{remark}

\section{Applications}%
\label{sec:applications}

\subsection{Self-normalized sums of \emph{m}-dependent random variables}%
\label{sub:app1}

Let $d \geq 1$ and let $\mathbb{Z}^d$ denote the $d$-dimensional space of positive integers. For any $i = (i_1, \dots, i_d), j = (j_1, \dots, j_d) \in \mathbb{Z}^d$, we define the distance by $\lvert i - j \rvert \coloneqq \max_{1  \leq k \leq d} \lvert i_k -
j_k \rvert$, and for $A, B \subset \mathbb{Z}^d$, we define the distance between $A$ and $B$ by $\rho(A, B) = \inf \{ \lvert i - j  \rvert : i \in A, j \in B \}$.
Let $\mathcal{J}$ be a subset of $\mathbb{Z}^d$, and we say a field of random variables $\{ X_i : i \in \mathcal{J} \}$ is an $m$-dependent random field if $\{ X_i , i \in A \}$ and $\{ X_j , j \in B \}$ are independent whenever $\rho(A, B) > m$ for any $A, B
\subset \mathcal{J}$. 
Choose $A_i = \{ j : \lvert i - j \rvert \leq m \}$, $B_i = \{ j : \lvert i - j \rvert \leq 2m \}$. 
We have $\lvert \{ j : A_j \cap B_i \} \rvert \leq (6m+1)^d$ for all $i$. 
Applying \cref{thm1}, 
we have the following theorem. 

\begin{theorem}
	\label{thm-mdep}
	Let $\{ X_i , i \in \mathcal{J} \}$ be an $m$-dependent field with $\E  X_i  = 0$ and assume that $\E |X_i|^3 < \infty$. Let $Y_i = \sum_{j \in A_i} X_i$. Assume that $\sigma^2 \coloneqq \sum_{i \in \mathcal{J}} \E \{ X_i Y_i \} > 0$. Let $W$ be as in \cref{eq-SVW}. Then, 
	\begin{align*}
		\sup_{z \in \R} \bigl\lvert \P( W \leq z ) - \Phi(z) \bigr\rvert
		& \leq C (m+1)^{3d} (1 + (m+1)^{d} \lvert \mathcal{J} \rvert^{1/3} \gamma^{2/3})( \gamma + |\mathcal{J}|^{-1/2}), 
	\end{align*}
	where $\gamma = \sum_{i \in \mathcal{J}} \E |X_i|^3 / \sigma^3.$
\end{theorem}

\subsection{Graph dependency}%
\label{sub:app2}

We now consider a field of random variables $\{ X_i : i \in \mathcal{V} \}$ indexed by the vertices of a graph $\mathcal{G}= (\mathcal{V}, \mathcal{E})$. We say $\mathcal{G}$ is a dependency graph if for any pair of disjoint sets $\Gamma_1 $
and $\Gamma_2$ in $\mathcal{V}$ such that no edge has one endpoint in $\Gamma_1$ and the other in $\Gamma_2$, then the sets of random variables $\{ X_i : i \in \Gamma_1 \}$ and $\{ X_j : j \in \Gamma_2 \}$ are independent. Let $A_i = \{ j \in \mathcal{V} : \text{there is an edge connecting $i$ and $j$}\}$, $A_{ij} = A_i \cup A_j$, $B_i = \cup_{j \in A_i} A_j$. Noting that $\lvert \{ j : A_j \cap B_i \} \rvert \leq C d^{3}$, and applying \cref{thm1}, 
we have the following theorem. 

\begin{theorem}
	\label{thm-graph-dep}
	Let $\{ X_i , i \in \mathcal{V} \}$ be a field of random variables indexed by the vertices of a dependency graph $\mathcal{G} = (\mathcal{V}, \mathcal{E})$.  Assume that $\E  X_i  = 0$. Let $S = \sum_{i \in \mathcal{V}} X_i $, $Y_i = \sum_{j \in A_i} X_j $ and let $V = \sqrt{ (
		\sum_{i \in \mathcal{V}} (X_i Y_i - \bar X \bar Y) )_{+} }$. Put $W = S / V$ and $n = \lvert \mathcal{V} \rvert$. Let $d$ be the maximal degree of $\mathcal{G}$. Then, we have 
	\begin{align*}
		\sup_{z \in \R} \lvert \P(W \leq z) - \Phi(z) \rvert
		& \leq C d^{9} ( 1 + d^{6} n^{1/3} \gamma^{2/3} ) \bigl( n^{-1/2} + \gamma  \bigr) ,
	\end{align*}
	where 
	\begin{math}
		\gamma =  \sum_{i \in \mathcal{V}}\E \lvert X_i / \sigma\rvert^3 ,
	\end{math}
	and $\sigma^2 = \sum_{i \in \mathcal{V}} \E \{ X_i Y_i \}$.
\end{theorem}
Graph dependency was firstly discussed in \cite{Bal89}, where a Berry--Esseen bound for non-self-normalized version of $W$ is also proved. However, the self-normalized Berry--Esseen bound is new. 

\section{Some preliminary lemmas and propositions}%
\label{sec:proof_of_theorem_1_1}
In this section, we first prove some preliminary lemmas, and then prove an important concentration inequality, which is of independent interest. 

\subsection{Some Preliminary lemmas}%
\label{sub:proof-main}

Let
\begin{equ}
	\bar X_i & \coloneqq X_i \I ( \lvert X_i \rvert \leq \sigma/\kappa), & \bar Y_i & \coloneqq \sum_{j \in A_i} \bar X_i, & \bar \sigma & = \Bigl( \sum_{j \in \mathcal{J}} \E \{ \bar X_j \bar Y_j \} \Bigr)^{1/2},\\
	\bar S & \coloneqq   \sum_{i\in \mathcal{J}} \bar X_i	, & \bar V &\coloneqq \psi \Bigl( \sum_{i \in \mathcal{J}} \bar X_i \bar Y_i \Bigr),  &
	\bar W & \coloneqq \bar S / \bar V,
    \label{eq-barS}
\end{equ}
where 
\begin{equ}
	\psi(x) = ((x \vee 0.25\sigma^2) \wedge 2\sigma^2)^{1/2}. 
    \label{eq-psi}
\end{equ}
The following lemma provides bounds for $\bar \sigma^2$ and some tail and moment inequalities for $\{ \bar X_i, i \in \mathcal{J} \}$. 
\begin{lemma}
	\label{lem-1}
	Assume that (LD1) holds. 
	We have 
	\begin{gather}
		\bigl\lvert \bar \sigma^2 - \sigma^2 \bigr\rvert 
		 \leq 3 \kappa\sigma^2 \beta_2, 
        \label{eq-l11}\\
		\E \biggl(\sum_{i \in \mathcal{J}} (\bar X_i \bar Y_i - \E \{ \bar X_i \bar Y_i \}) \biggr)^2  \leq 
			\kappa^2 \sigma^4 \beta_3,
		\label{eq-l13}\\
	\P \biggl\{ \biggl\lvert \sum_{i \in \mathcal{J}} (\bar X_i \bar Y_i - \E \{ \bar X_i \bar Y_i \}) \biggr\rvert \geq \frac{\sigma^2}{2} \biggr\} \leq 
			4 \kappa^2 \beta_3. 
	\label{eq-l14}	
	\end{gather}
	Specially, if $\beta_2 \leq 1/(150\kappa)$, we have 
	\begin{align}
		0.98 \sigma^2 \leq \bar \sigma^2 \leq 1.02 \sigma^2 . \label{eq-l16}
	\end{align}
\end{lemma}
\begin{proof}
	[Proof of \cref{lem-1}]
	We first prove \cref{eq-l11}. Note that for any $i,j \in \mathcal{J}$, 
	\begin{align*}
		\MoveEqLeft\lvert \E \{ X_i X_j \} - \E \{ \bar X_i \bar X_j \} \rvert\\
		& \leq \E \lvert (X_i - \bar X_i) \bar X_j \rvert + \E \lvert \bar X_i (X_j - \bar X_j) \rvert + \E \lvert (X_i - \bar X_i) (X_j - \bar X_j) \rvert\\
		& \leq \frac{\sigma}{\kappa} (\E \lvert X_i  \I ( \lvert X_i \rvert > \sigma/\kappa ) \rvert + \E \lvert X_j \I ( \lvert X_j \rvert > \sigma /\kappa) \rvert ) \\
		& \quad + \E \{ |X_i X_j| \I(|X_i| > \sigma/\kappa , \lvert X_j \rvert > \sigma/\kappa) \}\\
		& \leq 1.5 \bigl( \E \{|X_i|^2 \I( \lvert X_i \rvert > \sigma/\kappa )\} + \E \{|X_j|^2 \I( \lvert X_j \rvert > \sigma/\kappa )\} \bigr). 
	\end{align*}
	Thus, 
	\begin{align*}
		\bigl\lvert  \bar \sigma^2  - \sigma^2 \bigr\rvert
		& = \biggl\lvert \sum_{i \in \mathcal{J}} \sum_{j \in A_i} \E \{ X_i X_j \} - \sum_{i \in \mathcal{J}} \sum_{j \in A_i} \E \{ \bar X_i \bar X_j \} \biggr\rvert\\
		& \leq 1.5 \sum_{i \in \mathcal{J}} \sum_{j \in A_i} ( \E \{\lvert X_i \rvert^2 \I( \lvert X_i \rvert > \sigma/\kappa ) +  \lvert X_j \rvert^2 \I( \lvert X_j \rvert > \sigma/\kappa ) \} ) \\
		& \leq 1.5 \sum_{i \in \mathcal{J}} \sum_{j \in A_i}  \E \{\lvert X_i \rvert^2 \I( \lvert X_i \rvert > \sigma/\kappa ) \}   \\
		& \quad + 1.5 \sum_{j \in \mathcal{J}} \sum_{i : j \in A_i} \E \{ \lvert X_j \rvert^2 \I( \lvert X_j \rvert > \sigma/\kappa ) \} ) \\
		& \leq 3 \sigma^2 \kappa \beta_2 .  
	\end{align*}
	This proves \cref{eq-l11}. Specially, \cref{eq-l16} follows directly from \cref{eq-l11}.

	Now, by Cauchy's inequality, we have with $c = \kappa^2$, 
	\begin{equ}
		\E \{ \bar X_i^2 \bar Y_i^2 \}
		& \leq \frac{c}{2} \E \bar X_i^4 + \frac{\kappa^3}{2c} \sum_{j \in A_i} \E \bar X_j^4 = \frac{\kappa^2}{2} \E \bar X_i^4 + \frac{\kappa}{2} \sum_{j \in A_i} \E \bar X_j^4 \\
		& \leq \frac{\kappa^2 \sigma}{2 \tau} \E \bar X_i^3 + \frac{\kappa\sigma}{2\tau} \sum_{j \in A_i} \E \bar X_j^3, 
        \label{eq-EXYi2}
	\end{equ}
	where we used the fact that $\lvert \bar X_i \rvert \leq \sigma/\kappa$ in the last line. 
	Taking summation over $i \in \mathcal{J}$ on both sides of \cref{eq-EXYi2} yields 
	\begin{equ}
		\sum_{i \in \mathcal{J}} \E \{ \bar X_i^2 \bar Y_i^2 \}
		& \leq \kappa \sigma^4 \beta_3. 
        \label{eq-a1.1-aa}
	\end{equ}
	By \cref{eq-a1.1-aa}, we have the left hand side of \cref{eq-l13} is 
	\begin{align*}
		\MoveEqLeft[1]\E \biggl(\sum_{i \in \mathcal{J}}\sum_{j \in B_i} (\bar X_i \bar Y_i - \E \{ \bar X_i \bar Y_i \})(\bar X_j \bar Y_j - \E \{ \bar X_j \bar Y_j \}) \biggr)\\
		& =   \sum_{i \in \mathcal{J}}\sum_{j \in B_i} \biggl\{ \frac{1}{2} \E \{ \bar X_i^2 \bar Y_i^2 \} + \frac{1}{2} \E \{ \bar X_j^2 \bar Y_j^2 \} \biggr\} \\
		& =   \frac{1}{2} \biggl(\sum_{i \in \mathcal{J}}\sum_{j \in B_i}   \E \{ \bar X_i^2 \bar Y_i^2 \} + \sum_{j \in \mathcal{J}} \sum_{i : j \in B_i} \E \{ \bar X_j^2 \bar Y_j^2 \}  \biggr)\\
		& \leq \kappa^2\sigma^4\beta_3.	
	\end{align*}
	This proves \cref{eq-l13}. 
	Moreover, \cref{eq-l14} follow immediately from \cref{eq-l13} and the Markov inequality. 
\end{proof}

The following lemma will be useful in the proof of our main result. 
Let $\xi_i = \bar X_i / \bar V, \eta_i = \bar Y_i / \bar V$ and $W^{(i)} = W - \eta_i$. 
\begin{lemma}
	\label{lem2.2}
	Assume that \textup{(LD1)} holds and assume that $\beta_2 \leq 1/(150\kappa)$. Let $f$ be any absolutely continuous function such that $\lVert f \rVert \leq 1$ and $\lVert f' \rVert \leq 1$. We have 
	\begin{equ}
		\sum_{i \in \mathcal{J}} \bigl\lvert \E \{\xi_i f(\bar W - \eta_i)\} \bigr\rvert \leq 62 \kappa^2 \beta_3 + 2 \tau \beta_2. 
        \label{eq-lem22-1}
	\end{equ}
\end{lemma}
\begin{proof}
Let $N(A_i) = \{j: A_i \cap A_j \neq \emptyset\}$, $N(B_i) = \{ j: B_i \cap A_j \neq \emptyset \}$, and
\begin{align*}
	\bar V^{(i)} = \psi \biggl( \sum_{k \in A_i^c} \sum_{ l \in A_i^c \cap A_k } \bar X_k  \bar X_l  \biggr).
\end{align*}
The left hand side of \cref{eq-lem22-1} can be bounded by 
\begin{align*}
	\sum_{i \in \mathcal{J}} \bigl\lvert \E \{\xi_i f(\bar W - \eta_i)\} \bigr\rvert
	& \leq \sum_{i \in \mathcal{J}} \biggl\lvert \E \biggl\{ \frac{\bar X_i}{\bar V} f \Bigl( \frac{\bar S - \bar Y_i}{\bar V} \Bigr) \biggr\} -  \E \biggl\{ \frac{\bar X_i}{\bar V^{(i)}} f \Bigl( \frac{\bar S - \bar Y_i}{\bar V^{(i)}} \Bigr)
	\biggr\}\biggr\rvert\\
	& \quad + \sum_{i \in \mathcal{J}} \biggl\lvert \E \biggl\{ \frac{\bar X_i}{\bar V^{(i)}} f \Bigl( \frac{\bar S - \bar Y_i}{\bar V^{(i)}} \Bigr)
	\biggr\} \biggr\rvert \\
	& \coloneqq T_1 + T_2	. 
\end{align*}
Recall that $\bar V \geq \sigma / 2$ and $\bar V^{(i)} \geq \sigma / 2$. 
Then, it follows that 
\begin{equ}
	\biggl\lvert \frac{1}{\bar V} - \frac{1}{\bar V^{(i)}} \biggr\rvert
	& \leq \frac{1}{\bar V \bar V^{(i)} (\bar V + \bar V^{(i)})} \biggl\lvert \sum_{j \in A_i} \sum_{k \in A_j} \bar X_j \bar X_k +  \sum_{j \in A_i^c} \sum_{k \in A_i \cap A_j } \bar X_j \bar X_k \biggr\rvert\\
	& \leq \frac{4}{\sigma^3} \biggl(\sum_{j \in A_i} \sum_{k \in A_j} \lvert \bar X_j \bar X_k \rvert + \sum_{j \in A_i^c} \sum_{k \in A_j \cap A_i} \lvert \bar X_j \bar X_k \rvert\biggr) .
    \label{eq-VVi}
\end{equ}
By \cref{eq-VVi}, 
\begin{equ}
	T_1
	& \leq \sum_{i \in \mathcal{J}} \E \biggl\{\biggl( 1 + \frac{2\lvert \bar S - \bar Y_i \rvert}{\sigma} \biggr) \lvert \bar X_i \rvert \biggl\lvert \frac{1}{\bar V} - \frac{1}{\bar V^{(i)}} \biggr\rvert\biggr\}\\
	& \leq \frac{4}{\sigma^3} \sum_{i \in \mathcal{J}} \sum_{j \in A_i} \sum_{k \in A_j}\E \biggl\{\biggl( 1 + \frac{2\lvert \bar S - \bar Y_i \rvert}{\sigma} \biggr) \lvert \bar X_i \bar X_j \bar X_k \rvert \biggr\}\\
	& \quad + \frac{4}{\sigma^3} \sum_{i \in \mathcal{J}} \sum_{k \in A_i} \sum_{j : k \in A_j}\E \biggl\{\biggl( 1 + \frac{2\lvert \bar S - \bar Y_i \rvert}{\sigma} \biggr) \lvert \bar X_i \bar X_j \bar X_k \rvert \biggr\}\\
	& \leq \frac{8 \kappa^2}{3\sigma^3} \sum_{i \in \mathcal{J}}\E \biggl\{ \biggl( 1 + \frac{2\lvert \bar S - \bar Y_i \rvert}{\sigma} \biggr) \lvert \bar X_i^3 \rvert\biggr\} \\
	& \quad  + \frac{8 \kappa}{3\sigma^2} \sum_{j \in \mathcal{J}} \sum_{i : j \in A_i} \E \biggl\{  \biggl( 1 + \frac{2\lvert \bar S - \bar Y_i \rvert}{\sigma} \biggr) \lvert \bar X_j^3 \rvert		 \biggr\} \\
	& \quad + \frac{8}{3\sigma^2} \sum_{k \in \mathcal{J}} \sum_{j : k \in A_j } \sum_{i: j \in A_i} \E \biggl\{  \biggl( 1 + \frac{2\lvert \bar S - \bar Y_i \rvert}{\sigma} \biggr) \lvert \bar X_k^3 \rvert		 \biggr\} . 
    \label{eq-VVz}
\end{equ}
Let 
\begin{math}
	\bar Y_{ij} = \sum_{k \in A_i \cup A_j} \bar X_k. 
\end{math}
Then, by $\lvert \bar Y_{ij} - \bar Y_i \rvert \leq \sum_{k \in A_j} \lvert \bar X_k \rvert \leq \sigma $, we have 
\begin{equ}
	\E \biggl\{ \biggl( 1 + \frac{2\lvert \bar S - \bar Y_i \rvert}{\sigma} \biggr) \lvert \bar X_i \rvert^3	 \biggr\} 
	& = \E \biggl\{ \biggl( 1 + \frac{2\lvert \bar S - \bar Y_i \rvert}{\sigma} \biggr) \biggr\} \E \{ \lvert \bar X_i \rvert^3\}, \\
	\E \biggl\{ \biggl( 1 + \frac{2\lvert \bar S - \bar Y_i \rvert}{\sigma} \biggr) \lvert \bar X_j \rvert^3	 \biggr\} 
	& \leq \E \biggl\{ \biggl( 1 + \frac{2\lvert \bar S - \bar Y_{ij} \rvert}{\sigma} \biggr) \biggr\} \E \{ \lvert \bar X_j \rvert^3\} +  2  \E \lvert \bar X_j \rvert^3 .
    \label{eq-lem22-21}
\end{equ}
As $|\bar Y_i| \leq  \sigma, |\bar Y_{ij}| \leq 2\sigma$ and $\E |\bar S| \leq \bar\sigma \leq 1.02 \sigma$, we have  
\begin{equ}
	\E \lvert \bar S - \bar Y_i \rvert 
	\leq  2.02 \sigma, \quad \E \lvert \bar S - \bar Y_{ij} \rvert \leq   3.02 \sigma. 
    \label{eq-lem22-22}
\end{equ}
Substituting \cref{eq-lem22-21,eq-lem22-22} to \cref{eq-VVz} yields 
\begin{equ}
	T_{1} \leq 62 \kappa^2 \beta_3.
    \label{eq-R1}
\end{equ}

Moreover, observe that 
\begin{equ}
	\sum_{i \in \mathcal{J}} \lvert \E  \bar X_i  \rvert
	& \leq \sum_{i \in \mathcal{J}} \E \{ \lvert \bar X_i \rvert \I( \lvert \bar X_i \rvert > \sigma / \kappa ) \} \\
	& \leq \frac{\kappa}{\sigma} \sum_{i \in \mathcal{J}} \E \{ \lvert X_i \rvert^2 \I( \lvert X_i \rvert \geq \sigma / \kappa ) \} = \kappa \sigma \beta_2. 
    \label{eq-beta22}
\end{equ}
For $T_{2}$, as $\bar V^{(i)} \geq \sigma / 2$ and $\|f\| \leq 1$, by \cref{eq-beta22}, we have 
\begin{equ}
	\lvert T_{2} \rvert 
	& \leq \frac{2}{\sigma} \sum_{i \in \mathcal{J}} \lvert \E \{ \bar X_i \} \rvert  \leq 2 \kappa\beta_2. 
    \label{eq-R2}
\end{equ}
Combining \cref{eq-R1,eq-R2}, we complete the proof. 	
\end{proof}
The next lemma provides upper bounds for the fourth moments of $\sum_{i \in \mathcal{J}} \bar X_i$ and $\sum_{i \in \mathcal{J}} \bar Y_i$. 

\begin{lemma}
	\label{lem-4moment}
	Under (LD1) and (LD2), and assume that $\beta_2 \leq 1/(150 \kappa)$. 
	We have 
	\begin{align}
		\E \biggl( \sum_{i \in \mathcal{J}} \bar X_i \biggr)^4 & \leq 1161 (\theta + 1)\sigma^4,
		\label{eq-l3.3-1}
		\\
		\E \biggl( \sum_{i \in \mathcal{J}} \bar Y_i\biggr)^4 & \leq  1161 \kappa^2(\theta + 1)\sigma^4.
		\label{eq-l3.3-2}
	\end{align}
	As a consequence, 
	\begin{align}
		\E \biggl\{\biggl( \sum_{i \in \mathcal{J}} \bar X_i \biggr)^2 \biggl( \sum_{j \in \mathcal{J}} \bar Y_j\biggr)^2\biggr\} 
		& \leq 1161 \kappa (\theta + 1)\sigma^4.
		\label{eq-l3.3-3}
	\end{align}
\end{lemma}
\begin{proof}

We first prove the first inequality. Recall that $\E X_i = 0$ and thus
\begin{equ}
|\E \bar S| & = \biggl\lvert \sum_{i \in \mathcal{J}} \E  \bar X_i  \biggr\rvert
			   = \biggl\lvert \sum_{i \in \mathcal{J}} \E \{ X_i \I( |X_i| > \sigma/\kappa ) \} \biggr\rvert \leq \beta_2 \kappa\sigma. 
    \label{eq-Smean}
\end{equ}
Let $\bar X_{i,0} = \bar X_i - \E \{ \bar X_i \}$ and let $\bar Y_{i,0} = \sum_{j \in A_i} \bar X_{j,0}$. Under (LD1), we have 
\begin{align*}
		\E \lvert \bar S - \E \bar S \rvert^4 
		& = \sum_{i \in \mathcal{J}} \E \biggl\{ \bar X_{j,0} \biggl( \Bigl( \sum_{i \in \mathcal{J}} \bar X_{j,0} \Bigr)^3 - \Bigl( \sum_{i \in A_i^c} \bar X_{j,0} \Bigr)^3 \biggr) \biggr\}\\
		& = 3 \sum_{j \in \mathcal{J}} \sum_{j \in A_i} \E \biggl\{ \bar X_{i,0} \bar X_{j,0} \Bigl( \sum_{k \in \mathcal{J}} \bar X_{k,0} \Bigr)^2 \biggr\}\\
		& \quad - 3 \sum_{i \in \mathcal{J}} \sum_{j \in A_i} \sum_{k \in A_i} \sum_{l \in \mathcal{J}} \E \{ \bar X_{i,0} \bar X_{j,0} \bar X_{k,0} \bar X_{l,0} \} \\
		& \quad + \sum_{i \in \mathcal{J}} \sum_{j \in A_i} \sum_{k \in A_i} \sum_{ l \in A_i } \E \{ \bar X_{i,0} \bar X_{j,0} \bar X_{k,0} \bar X_{l,0} \} \\
		& \leq T_1 + T_2 ,  
\end{align*}
where 
\begin{align*}
	T_1 & = \frac{9}{2} \sum_{j \in \mathcal{J}} \E \biggl\{ |\bar X_{i,0} \bar Y_{i,0}| \Bigl( \sum_{k \in \mathcal{J}} \bar X_{k,0} \Bigr)^2 \biggr\}, \\
	T_2 & = \frac{5}{2} \sum_{i \in \mathcal{J}} \sum_{j \in A_i} \sum_{k \in A_i} \sum_{ l \in A_i } \E \{ \lvert \bar X_{i,0} \bar X_{j,0} \bar X_{k,0} \bar X_{l,0} \rvert \}. 
\end{align*}
By (LD2), we have 
\begin{align*}
	\begin{split}
		|T_1| 
		& \leq 9 \sum_{i \in \mathcal{J}} \E \biggl\{ |\bar X_{i,0} \bar Y_{i,0}| \Bigl( \sum_{k \in B_i^c} \bar X_{k,0} \Bigr)^2 \biggr\}+ 9  \sum_{i \in \mathcal{J}}  \E \biggl\{ |\bar X_{i,0} \bar Y_{i,0}| \Bigl( \sum_{k \in B_i} \bar X_{k,0} \Bigr)^2 \biggr\}\\
		& \coloneqq T_{11} + T_{12}. 
	\end{split}
\end{align*}
For $T_{11}$, noting that
$(\bar X_{i,0}, \bar Y_{i,0})$ is independent of $\{ \bar Y_{k,0} : j \in B_i^c \}$, then 
\begin{equ}
	T_{11} & = 9 \sum_{i \in \mathcal{J}}  \E \lvert \bar X_{i0} \bar Y_{j,0} \rvert \times  \E \biggl\{  \Bigl( \sum_{k \in B_i^c} \bar X_{k,0} \Bigr)^2 \biggr\}. 
    \label{eq-T11-1}
\end{equ}
Note that $|\bar X_{i,0}| \leq 2\sigma/\kappa$. By \cref{lem-1}, the second expectation of \cref{eq-T11-1} can be bounded by 
\begin{equ}
	\E \biggl\{  \Bigl( \sum_{k \in B_i^c} \bar X_{k,0} \Bigr)^2 \biggr\}
	& \leq 2 \E \biggl\{ \Bigl( \sum_{k \in \mathcal{J}} \bar X_{k,0} \Bigr)^2 \biggr\} + 2 \E \biggl\{ \Bigl( \sum_{k \in B_i} \bar X_{k,0} \Bigr)^2 \biggr\}\\
	& \leq 2 \E |\bar S|^2 + 8  \sigma^2\leq 10.04\sigma^2.
    \label{eq-T11-2}
\end{equ}
Recalling that $\theta = \sum_{i \in \mathcal{J}}\sum_{j \in A_i} \E \lvert \bar X_i \bar X_j/\sigma^2 \rvert$, and noting that $\lvert \bar Y_i \rvert \leq 1$, we have the first expectation of \cref{eq-T11-1} is bounded by 
\begin{equ}
	\sum_{i \in \mathcal{J}}  \E |\bar X_{i,0} \bar Y_{i,0}|
	& \leq \sum_{i \in \mathcal{J}} (\E |X_i Y_i| + 3 \kappa| \E \bar X_i| ) \\
	& \leq \theta \sigma^2 + 3\sigma^2 \kappa \beta_2 \leq (\theta + 0.02) \sigma^2, 
    \label{eq-xybound}
\end{equ}
where we used \cref{eq-beta22}  and the assumption that $\beta_2 \leq 1/(150\kappa)$ in the last line. 
Substituting \cref{eq-T11-2,eq-xybound} to \cref{eq-T11-1} gives 
\begin{equ}
	T_{11} & \leq 99(\theta + 1)\sigma^4. 
    \label{eq-T1}
\end{equ}
For $T_{12}$, as 
$\sum_{k \in B_i}|\bar X_{k,0}| \leq 2 \sigma$, we have $(\sum_{k \in B_i}|\bar X_{k,0}|)^2 \leq 4 \sigma^2$ almost surely, and thus
\begin{equ}
	T_{12}  
		& \leq 36 \sum_{i \in \mathcal{J}} \E |\bar X_{i,0} \bar Y_{i,0}| \leq 36 (\theta + 1) \sigma^4 . 
    \label{eq-T2}
\end{equ}
By \cref{eq-xybound} and noting that $(\sum_{k \in A_i} \bar X_{k,0})^2 \leq 4 \sigma^2$ almost surely,  we have  
\begin{equ}
	|T_2| \leq 10 (\theta + 1)\sigma^4 . 
    \label{eq-T34}
\end{equ}
By \cref{eq-Smean,eq-T1,eq-T2,eq-T34}, and recalling that $\beta_2 \leq 1/(150 \kappa^2)$ and $\kappa \geq 1$, we have 
	\begin{equ}
		\E |\bar S|^4 
		& \leq 8 \bigl( \E \lvert \bar S - \E \{ \bar S \} \rvert^4 + \lvert \E \bar S \rvert^4 \bigr)\\
		& \leq 8 \bigl( 145 (\theta + 1) \sigma^4 + \beta_2^4 \sigma^4\bigr) \\
		& \leq 1161  (\theta + 1) \sigma^4. 
        \label{eq-ES41}
	\end{equ}
	This proves \cref{eq-l3.3-1}.

For the second inequality, observe that 
\begin{align*}
	\sum_{i \in \mathcal{J}} \bar Y_i
	& = \sum_{i \in \mathcal{J}} \sum_{j \in A_i} \bar X_j = \sum_{j \in \mathcal{J}} \lambda_j \bar X_j , 
\end{align*}
where 
\begin{math} 
	\lambda_j \coloneqq \lvert \{ i : j \in A_i \} \rvert. 
\end{math}
Then, we have $\lambda_j \leq \kappa$ for all $j \in \mathcal{J}$. 
Using a similar argument, we have 
\cref{eq-l3.3-2} holds. The inequality \cref{eq-l3.3-3} follows from \cref{eq-l3.3-1} and \cref{eq-l3.3-2} by applying Cauchy's inequality. 
\end{proof}

\subsection{Concentration inequalities}%
\label{seb:proof_of_proposition}
In this subsection, we prove a concentration inequality under (LD1) and (LD2). 
We denote by $C, C_1, C_2, \dots$ absolute constants. 
\newcounter{mycounter} 
\setcounter{mycounter}{0}
\newcommand\showmycounter{\stepcounter{mycounter}\themycounter}
\newcommand\Const{\stepcounter{mycounter}C_{\arabic{mycounter}}}
\begin{proposition}
	\label{lem-con2}
	Assume that (LD1) and (LD2) holds and assume that $\beta_2 \leq 1/(150\kappa)$ and  $\beta_3 \leq 1/(150 \kappa^2)$. Let $\mathcal{F}_i = \sigma(X_j : j \in A_i)$. For any $z \in \R$ and ${\mathcal{F}_i}$-measurable
	random variables $a$ and $b$ such that $a \leq b$, we have 
	\begin{align*}
		\MoveEqLeft
		\P^{{\mathcal{F}_i}} \Bigl(z + \frac{a}{\bar V} \leq \bar W \leq z + \frac{b}{\bar V}\Bigr) \\
		& \leq \frac{2(b - a )}{\sigma} +  \Const   \kappa^2 \biggl( \beta_3 + \sum_{j \in N(B_i)} \bigl(\E^{{\mathcal{F}_i}} | \bar X_j/\sigma|^3 + \E^{{\mathcal{F}_i}} \lvert \bar S \bar X_j^3/\sigma^4 \rvert\bigr)\biggr).
	\end{align*}
	where $\P^{{\mathcal{F}_i}}$ and $\E^{\mathcal F_i}$ denote the conditional probability and conditional expectation given ${\mathcal{F}_i}$, respectively. 
\end{proposition}

\begin{proof}
	[Proof of \cref{lem-con2}]
	We use the idea of Proposition 3.2 in \cite{Che04b} to prove this proposition. 
	Let $\alpha = 20\kappa^2 \sigma \beta_3$. Without loss of generality, we assume that $b - a \leq \sigma/2$, otherwise the result is trivial. As $\kappa^3 \geq 1$ and by assumption we have $\kappa^2\beta_3 \leq 1/150$. Therefore, we have 
	\begin{equ}
		b - a + \alpha \leq 0.8 \sigma. 
        \label{eq-babound}
	\end{equ}
	Let 
	\begin{equation}
		\label{eq-2.1}
		f_{a,b,v}(w) =  
		\begin{dcases}
			- \frac{b - a + \alpha}{2v} & \text{if $w \leq z + a/v - \alpha/v$}, \\
			\frac{1}{2\alpha} \Bigl(w - \frac{\alpha + \alpha}{v}\Bigr)^2 - \frac{b - a + \alpha}{2v} & \text{if $z + \frac{a - \alpha}{v} < w \leq z + \frac{a}{v}$}, \\
			w - \frac{a + b}{2v} & \text{if $z + a/v < w \leq z + b/v$}, \\
			- \frac{1}{2\alpha} \Bigl(w - \frac{b - \alpha}{v}\Bigr)^2 + \frac{b - a + \alpha}{2v} & \text{if $z + \frac{b}{v} < w \leq z + \frac{b + \alpha}{v}$},\\
			\frac{b - a + \alpha}{2v} & \text{if $w > z + (b + \alpha)/v$}.
		\end{dcases}
	\end{equation}
	Then, we have 
	\begin{align}
		\label{eq-fFv}
		f_{a,b,v}'(w) = 
		\begin{dcases}
			1 & \text{if $z + a/v \leq w \leq z + b/v$}, \\
			0 & \text{if $w < z + a/v - \alpha/v$ or $w > z + b/v	 + \alpha/v$}, \\
			\text{linear} & \text{otherwise}. 
		\end{dcases}
	\end{align}
	Let $f:= f_{a,b, \bar V}$.
	Then, it follows from $\bar V \geq \sigma / 2$ that $\lVert f \rVert \leq (b - a + \alpha)/(2\bar V) \leq (b - a + \alpha)/\sigma$.
	Fix $i$. 
Recall that $N(A_i) = \{j: A_i \cap
A_j \neq \emptyset\}$ and $N(B_i) = \{ j: B_i \cap A_j \neq \emptyset \}$.
	Let 
	\begin{math}
		\sigma_i = ( \E \{ \sum_{j \in N(B_i)^c} \bar X_j \}^2 )^{1/2}.
	\end{math}
	Then, recalling that $\beta_3 \leq 1/(150 \kappa^2)$ and by H\"older's inequality, we have 
	\begin{align*}
		\sigma_i & \leq \biggl( \lvert N(B_i) \rvert^{2} \sum_{j \in N(B_i)} \E \lvert \bar X_j \rvert^3 \biggr)^{1/3}\leq \bigl( \kappa^2 \sigma^3 \beta_3 \bigr)^{1/3} \leq 0.2 \sigma. 
	\end{align*}
	On one hand, as $\|f\| \leq (b - a + \alpha)/\sigma$, $\bar V \geq \sigma/2$, and $\{ X_j : j \not\in N(B_i) \}$ is independent of ${\mathcal{F}_i}$, we have 
	\begin{equ}
		\biggl\lvert \E^{{\mathcal{F}_i}} \biggl\{ \biggl( \sum_{j \not\in N(B_i)} \xi_j \biggr) f(\bar W)  \biggr\}
		\biggr\rvert		& \leq 
		\frac{2(b - a + \alpha)	}{\sigma^2} \E \biggl\lvert \sum_{j \not\in N(B_i)} \bar X_i \biggr\rvert \\
							& \leq \frac{2\sigma_i}{\sigma^2} (b - a + \alpha) \leq \frac{b - a + \alpha}{2\sigma}.
        \label{eq-upper1}
	\end{equ}
	Let 
	\begin{align*}
		\hat M_j (t) & = \bar X_j \bigl( \I( - \bar Y_j \leq t \leq 0 ) - \I(0 < t \leq - \bar Y_j)\bigr), & M_j(t) & = \E \{ \hat M_j(t) \}, \\
		\hat M(t) & = \sum_{j \in N(B_i)^c} \hat M_j(t), \quad  & M(t) & = \E \{ \hat M(t) \}.
	\end{align*}
	For the lower bound of the left hand side of \cref{eq-upper1}, observe that 
	\begin{equ}
		\E^{{\mathcal{F}_i}} \biggl\{ \biggl( \sum_{j \not\in N(B_i)} \xi_j \biggr) f(\bar W)  \biggr\}
		& = \sum_{j \not\in N(B_i)} \E ^{{\mathcal{F}_i}}\bigl\{ \xi_j (f(\bar W) - f(\bar W - \eta_j) \bigr\} \\
		& \quad + \sum_{j \not\in N(B_i)} \E^{{\mathcal{F}_i}} \{ \xi_j f(\bar W -  \eta_j) \}\\
		& \coloneqq Q_1 + Q_2 + Q_3 + Q_4 ,  
        \label{eq-lower}
	\end{equ}
	where 
	\begin{align*}
		Q_1 & = \E^{{\mathcal{F}_i}} \biggl\{ \frac{f'(\bar W)}{\bar V^2} \int_{-\infty}^{\infty} {M}(t) dt \biggr\}, \\
		Q_2 & = \E^{{\mathcal{F}_i}} \biggl\{ \frac{1}{\bar V^2}\int_{-\infty}^{\infty} \biggl( f'\Bigl( \frac{\bar S + t}{\bar V} \Bigr) - f'\Bigl( \frac{\bar S}{\bar V} \Bigr) \biggr) M(t) dt \biggr\}, \\
		Q_3 & = \E^{{\mathcal{F}_i}} \biggl\{ \frac{1}{\bar V^2}\int_{-\infty}^{\infty} f'\Bigl( \frac{\bar S + t}{\bar V} \Bigr) \bigl( \hat M(t) - {M}(t) \bigr) dt \biggr\}, \\
		Q_4 & = \sum_{j \in N(B_i)^c} \E^{{\mathcal{F}_i}} \bigl\{ \xi_j f(\bar W - \eta_j) \bigr\}. 
	\end{align*}
	We now bound $Q_1$ to $Q_4$ one by one. 

	Observe that
	\begin{equ}
		\int_{-\infty}^{\infty} M(t)dt
		& = \sum_{j \in N(B_i)^c} \E \{ \bar X_j \bar Y_j \}= \sigma^2 - \sum_{j \in N(B_i)} \E \{ \bar X_j \bar Y_j \}.
        \label{eq-Mint}
	\end{equ}
	Now, by H\"older's inequality, the second term of the R.H.S. of \cref{eq-Mint} is bounded by 
	\begin{align*}
		\biggl\lvert \sum_{j \in N(B_i)} \E \{ \bar X_j \bar Y_j \} \biggr\rvert
		& \leq \biggl( \lvert N(B_i) \rvert^{1/2} \sum_{j \in N(B_i)} \E \lvert \bar X_j \bar Y_j \rvert^{3/2} \biggr)^{2/3}\\
		& \leq \biggl( \kappa^{1/2} \sum_{j \in N(B_i)} \biggl[ \frac{\kappa^{3/2}}{2} \E \lvert \bar X_j \rvert^3 + \frac{1}{2 \kappa^{3/2}} \E \lvert \bar Y_j \rvert^3 \biggr] \biggr)^{2/3}\\
		& \leq ( \kappa^2 \beta_3 )^{2/3} \sigma^2 \leq (1/150)^{2/3} \sigma^2 \leq 0.1 \sigma^2.
	\end{align*}
	Thus, we have 
	\begin{align*}
		\biggl\lvert \int_{-\infty}^{\infty} M(t) dt - \sigma^2 \biggr\rvert \leq 0.1 \sigma^2.
	\end{align*}
	Note that $\sigma^2/4 \leq \bar V^2 \leq 2\sigma^2$.  
	For $Q_1$, we have 
	\begin{equ}
		Q_1 & \geq  0.9 \E^{{\mathcal{F}_i}} \biggl\{\frac{\sigma^2}{\bar V^2}\I ( z + a/\bar V \leq \bar W \leq a + b/\bar V )\biggr\} \\
			& \geq 0.45 \P^{{\mathcal{F}_i}} \bigl(  z + a/\bar V \leq \bar W \leq a + b/\bar V  \bigr) .
        \label{eq-Q1}
	\end{equ}
	For $Q_3$, 
	\begin{align*}
		\lvert Q_3 \rvert & \leq Q_{31} + Q_{32}, 
	\end{align*}
	where 
	\begin{align*}
		Q_{31} & = \E^{{\mathcal{F}_i}} \int_{|t| \leq \sigma} \frac{1}{\bar V^2}f'\Bigl( \frac{\bar S + t}{\bar V} \Bigr)  \lvert \hat M(t) -  M(t) \rvert dt , \\
		Q_{32} & = \E^{{\mathcal{F}_i}} \int_{|t| > \sigma} \frac{1}{\bar V^2} \lvert \hat M(t) - M(t) \rvert dt. 
	\end{align*}
	For $Q_{31}$, noting that $\|f'\|\leq 1$, we have 
	\begin{align*}
		\frac{1}{v}\int_{|t| \leq \sigma} \biggl\lvert f'\biggl( \frac{s + t}{v} \biggr)  \biggr\rvert^2 dt 
		& \leq 
		\int_{|t| \leq \sigma/v} f'\biggl( \frac{s}{v} + t \biggr)   dt  \leq {b - a + \alpha}, 
	\end{align*}
	and thus, by Cauchy's inequality, 
	\begin{equ}
		Q_{31} & \leq \frac{1}{8} \E^{{\mathcal{F}_i}} \int_{|t| \leq \sigma} \biggl\lvert \frac{1}{\bar V^2} f'\biggl( \frac{\bar S + t}{\bar V} \biggr) \biggr\rvert^2 dt + 2 \E^{{\mathcal{F}_i}} \int_{|t| \leq \sigma} \frac{1}{\bar V^2}\lvert \hat M(t) -  M(t) \rvert^2 dt \\
			   & \leq \frac{b - a + \alpha}{4 \sigma} + \frac{8}{\sigma^2}  \int_{|t| \leq \sigma} \E\lvert \hat M(t) - M(t) \rvert^2 dt ,
        \label{eq-Q310}
	\end{equ}
	where we used the fact that $\bar V \geq \sigma/2$ and $\hat M(t)$ is independent of ${\mathcal{F}_i}$. 
	For the second term of \cref{eq-Q310}, 
	letting $(\bar X_j^*, \bar Y_j^*)$ be an independent copy of $(\bar X_j, \bar Y_j)$, we have 
	\begin{equ}
		\MoveEqLeft \frac{1}{\sigma^2}  \int_{|t| \leq \sigma} \E\lvert \hat M(t) - M(t) \rvert^2 dt\\
		& = \frac{1}{\sigma^2} \sum_{j \in N(B_i)^c} \sum_{l \in N(B_j) \cap N(B_i)^c} \int_{|t| \leq \sigma} \Cov \bigl\{ \hat M_j (t), \hat M_l(t)\bigr\} dt\\
		& = \frac{1}{\sigma^3} \sum_{j \in N(B_i)^c} \sum_{l \in N(B_j) \cap N(B_i)^c} \E \bigl\{ \bar X_j \bar X_l ( |\bar Y_j| \wedge \lvert \bar Y_l \rvert ) 1 (\bar Y_j \bar Y_l \geq 0)  \\
		& \hspace{5cm} - \bar X_j \bar X_l^* ( |\bar Y_j| \wedge \lvert \bar Y_l^* \rvert ) 1 (\bar Y_j \bar Y_l^* \geq 0)\bigr\} \\
		& \leq \frac{1}{\sigma^3}\sum_{j \in \mathcal{J}} \sum_{l \in N(B_i)} \E \bigl\{ \lvert \bar X_j \bar X_l \rvert ( \lvert \bar Y_j \rvert \wedge \lvert \bar Y_l \rvert ) + \lvert \bar X_j \bar X_l^* \rvert ( \lvert \bar Y_j \rvert \wedge \lvert \bar Y_l^* \rvert ) \bigr\}\leq 2 \kappa^2 \beta_3. 
        \label{eq-Q311}
	\end{equ}
	Substituting \cref{eq-Q311} to \cref{eq-Q310} yields  
	\begin{equ}
		Q_{31} & \leq \frac{0.25 (b - a + \alpha)}{\sigma} + 16 \kappa^2 \beta_3. 
        \label{eq-Q31}
	\end{equ}
	For $Q_{32}$, 
	noting that 
	\begin{equ}
		\E \int_{-\infty}^{\infty} |t \hat M(t)| dt 
		& \leq \frac{1}{2} \sum_{i \in \mathcal{J}} \E \{ |\bar X_i \bar Y_i^2| \}\leq \frac{\kappa^2}{2} \sum_{i \in \mathcal{J}} \E \lvert \bar X_i \rvert^3, 
        \label{eq-Mhat-bound}
	\end{equ}
	we have 
	\begin{equ}
		Q_{32} & \leq \frac{4}{\sigma^3}\E \int_{|t| > \sigma} |t| \lvert \hat M(t) - M (t)\rvert dt\leq \frac{8}{\sigma^3} \E \int_{-\infty}^{\infty} |t\hat M(t)|dt\leq 4 \kappa^2 \beta_3. 
        \label{eq-Q32}
	\end{equ}
	Combining \cref{eq-Q31,eq-Q32}, 
	we have 
	\begin{equ}
		\lvert Q_3 \rvert
		& \leq \frac{0.125(b - a + \alpha)}{ \sigma} + 20 \kappa^2 \beta_3. 
        \label{eq-Q3}
	\end{equ}
	Let $f_{a,b, \bar V^{(j)}}$ be defined as in \cref{eq-fFv} by taking $v = \bar V^{(j)}$.
	For $Q_{4}$, we have 
	\begin{align*}
		Q_4 
		& = \sum_{j \in N(B_i)^c} \E^{{\mathcal{F}_i}} \biggl\{ \bar X_j \biggl[ \frac{1}{\bar V} f \biggl( \frac{\bar S - \bar Y_j}{\bar V}\biggr) - \frac{ 1 }{\bar V^{(j)}}f \biggl( \frac{\bar S - \bar Y_j}{\bar V^{(j)}}  \biggr) \biggr] \biggr\}\\
		& \quad + \sum_{j \in N(B_i)^c} \E^{{\mathcal{F}_i}} \biggl\{ \frac{\bar X_j}{ \bar V^{(j)} } f \biggl( \frac{\bar S - \bar Y_j}{\bar V^{(j)}} \biggr) -  \frac{\bar X_j}{ \bar V^{(j)} } f_{a,b, \bar V^{(j)}} \biggl( \frac{\bar S - \bar Y_j}{\bar V^{(j)}} \biggr) \biggr\} \\
		& \quad + \sum_{j \in N(B_i)^c} \E^{{\mathcal{F}_i}} \biggl\{  \frac{\bar X_j}{ \bar V^{(j)} } f_{a,b, \bar V^{(j)}} \biggl( \frac{\bar S - \bar Y_j}{\bar V^{(j)}} \biggr) \biggr\}\\
		& := Q_{41} + Q_{42} + Q_{43}	. 
	\end{align*}
	For any $j \in N(B_i)^c$, by \cref{eq-VVi,eq-VVz}, we have 
\begin{equ}
	Q_{41} 
	& \leq  \sum_{j \in N(B_i)^c }\E^{{\mathcal{F}_i}}\biggl\lvert   \biggl\{ \frac{\bar X_j}{\bar V} f \Bigl( \frac{\bar S - \bar Y_j}{\bar V} \Bigr) \biggr\} - \biggl\{ \frac{\bar X_j}{\bar V^{(j)}} f \Bigl( \frac{\bar S - \bar Y_j}{\bar V^{(j)}} \Bigr)
	\biggr\}\biggr\rvert\\
	& \leq \sum_{j \in N(B_i)^c }\E^{{\mathcal{F}_i}} \biggl\{\biggl( \|f\| + \frac{2\lvert \bar S - \bar Y_j \rvert}{\sigma} \biggr) \lvert \bar X_j \rvert \biggl\lvert \frac{1}{\bar V} - \frac{1}{\bar V^{(j)}} \biggr\rvert\biggr\}\\
	& \leq \frac{4}{\sigma^4} \sum_{j \in N(B_i)^c }\E^{{\mathcal{F}_i}} \biggl\{ \bigl( (b - a + \alpha) + {2\lvert \bar S - \bar Y_j \rvert} \bigr) \lvert \bar X_j \rvert \sum_{l \in N(A_j)} \lvert \bar X_l \bar Y_l \rvert \biggr\}\\
	& \leq \frac{4  (b - a + \alpha) }{3 \sigma^4}Q_{44} + \frac{8}{3 \sigma^4} Q_{45},  
        \label{eq-EF1}
	\end{equ}
	where 
	\begin{align*}
		Q_{44} & =  \sum_{j \in N(B_i)^c }\sum_{l \in N(A_j)} \sum_{k \in A_l} \E^{{\mathcal{F}_i}}\{ ( \lvert \bar X_j\rvert^3 + \lvert \bar X_l \rvert^3 + \lvert \bar X_k \rvert^3 )\}, \\
		Q_{45} & = \sum_{j \in N(B_i)^c }\sum_{l \in N(A_j)} \sum_{k \in A_l} \E^{{\mathcal{F}_i}} \bigl\{\bigl( {\lvert \bar S - \bar Y_j \rvert} \bigr) ( \lvert \bar X_j\rvert^3 + \lvert \bar X_l \rvert^3 + \lvert \bar X_k \rvert^3 )
	\bigr\}	. 
	\end{align*}
	Noting that $ \lvert N(B_i) \rvert \leq \kappa $, $\lvert N(A_i)\rvert \leq \kappa$, $\lvert \{ j: A_j \cap A_l \neq \emptyset \} \rvert \leq \kappa$ and $\lvert A_l \rvert \leq \kappa$, we have
	\begin{align}
		Q_{44}
		& \leq  \sum_{j \in N(B_i)^c} \sum_{l \in N(A_j)} \sum_{k \in A_l}\E^{{\mathcal{F}_i}} \lvert \bar X_j \rvert^3 
		+ \sum_{k \in N(B_i)^c} \sum_{l : k \in A_l} \sum_{j : A_j \cap A_l \neq \emptyset} \E^{{\mathcal{F}_i}} \lvert \bar X_k \rvert^3 \\
		& \quad + \sum_{l \in N(B_i)^c} \sum_{j: A_j \cap A_l \neq \emptyset} \sum_{k \in A_l} \E^{{\mathcal{F}_i}} \lvert \bar X_l \rvert^3  + \sum_{k \in N(B_i)} \sum_{l : k \in A_l} \sum_{j : A_j \cap A_l \neq \emptyset} \E^{{\mathcal{F}_i}} \lvert \bar X_k \rvert^3 \\
		& \quad + \sum_{l \in N(B_i)} \sum_{j: A_j \cap A_l \neq \emptyset} \sum_{k \in A_l} \E^{{\mathcal{F}_i}} \lvert \bar X_l \rvert^3 \\
		& \leq 3 \kappa^2 \sigma^3 \beta_3 + 2 \kappa^2 \sum_{j \in N(B_i)} \E^{{\mathcal{F}_i}} \lvert \bar X_j \rvert^3 , 
        \label{eq-EF0}
	\end{align}
	where we used the fact that $\E^{{\mathcal{F}_i}}|\bar X_j|^3 = \E \lvert \bar X_j \rvert^3$ for $j \in N(B_i)^c$.

Now, we bound $Q_{45}$.
	For any $l$, 
	let
	\begin{math}
		\bar S_l = \bar S - \sum_{k \in N(B_i) \cup A_j \cup A_l} \bar X_k, 
	\end{math}
	then $\bigl\lvert \lvert \bar S - \bar Y_j \rvert - \lvert \bar S_l \rvert \bigr\rvert \leq 2 \sigma$. Also, if $l \in N(B_i)^c$, we have $\bar S_l$ is independent of $\bar X_l$ and $(\bar S_l, \bar X_l)$ is independent of ${\mathcal{F}_i}$. 
	Noting that $\bar \sigma^2 \leq 1.02\sigma^2$ by \cref{eq-l16} and that $\kappa \geq 1$, we have 
	\begin{align*}
		\E \lvert \bar S_l \rvert 
		& \leq \E \lvert \bar S \rvert + 3 \kappa \sigma \leq \bar \sigma + 3 \sigma
		 \leq 4.02 \sigma. 
	\end{align*}
	Therefore, for $l \in N(B_i)^c$, 
	\begin{equ}
		\E^{{\mathcal{F}_i}} \lvert (\bar S - \bar Y_j) \bar X_l^3 \rvert
		& \leq \E \lvert \bar S_l \bar X_l^3 \rvert + 2  \sigma \E \lvert \bar X_l \rvert^3 \\
		& \leq \E \lvert \bar S_l \rvert \E \lvert \bar X_l \rvert^3 + 2 \sigma \E \lvert \bar X_l \rvert^3 \leq 6.02 \sigma \E \lvert \bar X_l \rvert^3. 
        \label{eq-linNB}
	\end{equ}
	For $l \in N(B_i)$, we have 
	\begin{equ}
		\E^{{\mathcal{F}_i}} \lvert (\bar S - \bar Y_j) \bar X_l^3 \rvert
		& \leq \sigma \E^{{\mathcal{F}_i}}\lvert \bar X_l \rvert^3 + \E^{{\mathcal{F}_i}}\lvert \bar S \bar X_l^3 \rvert.
        \label{lninNb}
	\end{equ}
	Hence, by \cref{eq-linNB,lninNb}, 
	we then obtain 
	\begin{equ}
		Q_{45}
		& \leq \E^{{\mathcal{F}_i}}\sum_{j \in N(B_i)^c} \sum_{l \in N(A_j)} \sum_{k \in A_l} { \lvert \bar S - \bar Y_j \rvert}  \lvert \bar X_j \rvert^3 \\
		& \quad + \E^{{\mathcal{F}_i}}\sum_{k \in N(B_i)^c}  \sum_{l: k \in A_l}\sum_{j :  A_j \cap A_l \neq \emptyset }  {\lvert \bar S - \bar Y_j \rvert} \lvert \bar X_k \rvert^3\\
		& \quad + \E^{{\mathcal{F}_i}}\sum_{l \in N(B_i)^c} \sum_{j : A_j \cap A_l \neq \emptyset } \sum_{k \in A_l} {\lvert \bar S - \bar Y_j \rvert} \lvert \bar X_l \rvert^3\\
		& \quad + \E^{{\mathcal{F}_i}}\sum_{k \in N(B_i)}  \sum_{l: k \in A_l}\sum_{j :  A_j \cap A_l \neq \emptyset }  {\lvert \bar S - \bar Y_j \rvert}\lvert \bar X_k \rvert^3\\
		& \quad + \E^{{\mathcal{F}_i}}\sum_{l \in N(B_i)} \sum_{j : A_j \cap A_l \neq \emptyset } \sum_{k \in A_l} {\lvert \bar S - \bar Y_j \rvert} \lvert \bar X_l \rvert^3\\
		& \leq \Const  \sigma \kappa^2 \biggl( \sum_{j \in \mathcal{J}} \E \lvert \bar X_j \rvert^3 + \sum_{j \in N(B_i)} \E^{{\mathcal{F}_i}} | \bar X_j|^3\biggr)  + \Const \kappa^2 \sum_{j \in N(B_i) } \E^{{\mathcal{F}_i}} \lvert \bar S \bar X_j^3 \rvert  .
        \label{eq-EFS}
	\end{equ}
	By \cref{eq-EF0,eq-EFS}, we have 
	\begin{equation}
		Q_{41} 
		\leq 
		 \Const  \kappa^2 \biggl( \beta_3 + \sum_{j \in N(B_i)} \E^{{\mathcal{F}_i}} | \bar X_j/\sigma|^3 + \sum_{j \in N(B_i) } \E^{{\mathcal{F}_i}} \lvert \bar S \bar X_j^3/\sigma^4 \rvert\biggr) .
        \label{eq-Q41}
	\end{equation}
	For $Q_{42}$, observing that 
	\begin{align*}
		\bigl\lvert f (w) - f_{a,b, \bar V^{(j)}}  \bigr\rvert \leq \biggl\lvert \frac{1}{\bar V} - \frac{1}{\bar V^{(j)}} \biggr\rvert, 
	\end{align*}
	and by \cref{eq-VVi,eq-EF0},
	we obtain 
	\begin{equ}
		Q_{42} & \leq \frac{2}{\sigma}\sum_{j \in N(B_i)^c} \E \biggl\{ \lvert \bar X_j \rvert \biggl\lvert \frac{1}{\bar V} - \frac{1}{\bar V^{(j)}} \biggr\rvert \biggr\} \\
			   & \leq \frac{8}{\sigma^3} \sum_{j \in N(B_i)^c} \sum_{k \in N(A_j)} \E^{{\mathcal{F}_i}} \bigl\{ \lvert \bar X_j  \bar X_k \bar Y_k \rvert \bigr\} \\
			   & \leq \frac{8}{3 \sigma^3}  \E^{{\mathcal{F}_i} } \sum_{j \in N(B_i)^c} \sum_{l \in N(A_j)} \sum_{k \in A_l}%
                      \biggl(  \lvert \bar{X}_j \rvert^3 + \lvert \bar{X}_k \rvert^3 + \lvert \bar{X}_l \rvert^3  \biggr) \\
			   & \leq 8 \kappa^2 \beta_3 + 8 \kappa^2 \sum_{j \in N(B_i)} \E^{{\mathcal{F}_i}} \lvert \bar X_j/ \sigma\rvert^3. 
        \label{eq-Q42}
	\end{equ}

	For $Q_{43}$, as $\beta_2 \leq 1/(150\kappa) \leq 1/ 150$, $j \in N(B_i)^c$, $V^{(j)} \geq 0.5 \sigma$ and $\|f_{a,b, \bar V^{(j)}}\| \leq (b - a + \alpha)/\sigma$, by \cref{eq-beta22}, we have 
	\begin{equ}
		\lvert Q_{43} \rvert
		& \leq \frac{2 \|f_{a,b, \bar V^{(j)}}\|}{\sigma} \sum_{j \in N(B_i)^c} \lvert \E^{{\mathcal{F}_i}}  \bar X_j  \rvert = \frac{2}{\sigma^2} (b - a + \alpha) \sum_{j \in N(B_i)^c} \lvert \E  \bar X_j  \rvert \\
		& \leq 2(b - a + \alpha)\beta_2/\sigma \leq 0.2 (b - a + \alpha)/\sigma. 
        \label{eq-Q43}
	\end{equ}
	By \cref{eq-Q41,eq-Q42,eq-Q43}, we have 
	\begin{equ}
		\lvert Q_4 \rvert & \leq 0.2 (b - a + \alpha)/\sigma  \\
						  & \quad + \Const  \kappa^2 \biggl( \beta_3 + \sum_{j \in N(B_i)} \E^{{\mathcal{F}_i}} | \bar X_j/\sigma|^3 + \sum_{j \in N(B_i) } \E^{{\mathcal{F}_i}} \lvert \bar S \bar X_j^3/\sigma^4 \rvert\biggr) .
		\label{eq-Q4}
	\end{equ}

	To bound $Q_2$, we note that 
	\begin{align*}
		 \MoveEqLeft \biggl( f'\Bigl( \frac{\bar{S} + t}{\bar{V}} \Bigr) - f'\Bigl( \frac{\bar{S}}{\bar{V}} \Bigr) \biggr)\\
		& = 
		\begin{cases}
			\frac{1}{\alpha} \int_0^t (\I\{ z + a/\bar{V} - \alpha/\bar{V} \leq (\bar{S} + s)/\bar{V} \leq z +  a / \bar{V} \} \\
			\hspace{1cm} - \I\{ z + b/\bar{V} \leq (\bar{S} + s)/ \bar{V}  \leq z + b/\bar{V} + \alpha/\bar{V} \}) ds & \text{if $t \geq 0$},\\
			- \frac{1}{\alpha} \int_{t}^0 (\I\{ z + a/\bar{V} - \alpha/\bar{V} \leq (\bar{S} + s)/\bar{V} \leq z +  a / \bar{V} \} \\
			\hspace{1cm} - \I\{ z + b/\bar{V} \leq (\bar{S} + s)/ \bar{V}  \leq z + b/\bar{V} + \alpha/\bar{V} \}) ds & \text{if $t < 0$}.
		\end{cases}
	\end{align*}
	Then, recalling that $\alpha = 20 \kappa^2\sigma\beta_3$, we have 
	\begin{equ}
		|Q_2| 
		& \leq \frac{4}{\sigma^2}\int_{|t| \geq \sigma} M(t) dt + \frac{8}{\alpha \sigma^2} \int_0^\sigma \int_0^t L(\alpha) ds |M(t)| dt \\
		& \quad + \frac{8}{\alpha \sigma^2} \int_{-\sigma}^0 \int_{t}^0 L(\alpha) ds |M(t)| dt \\
		& \leq \frac{8}{\sigma^3} ( \sigma\alpha^{-1}L(\alpha) + 1 ) \int_{-\infty}^{\infty} |tM(t)| dt \\
		& \leq 4 \kappa^2 (\sigma\alpha^{-1} L(\alpha) + 1) \beta_3\\
		& \leq 0.2L(\alpha) + 4 \kappa^2 \beta_3, 
        \label{eq-Q2}
	\end{equ}
	where 
	\begin{align*}
		L(\alpha) = \lim_{n \to \infty} \sup_{y,x \in \mathbb{Q}}
		\P^{{\mathcal{F}_i}} ( y + (x - 1/n)/\bar V  \leq \bar W \leq y +  (x + 1/n)/\bar V +  \alpha/\bar V	+ 1/n  ), 
	\end{align*}
	the notation $\mathbb{Q}$ denotes the set of rational numbers, and we applied \cref{eq-Mhat-bound} and Jansen's inequality in the last line. 

	Combining \cref{eq-upper1,eq-lower,eq-Q1,eq-Q2,eq-Q3,eq-Q4}, we have 
	\begin{equ}
		\MoveEqLeft
		0.45 \P^{{\mathcal{F}_i}} \bigl( z + a/\bar V \leq \bar W \leq z + b / \bar V \bigr)\\
		& \leq 0.9 \frac{b - a + \alpha}{\sigma} + 0.2 L(\alpha)\\
		& \quad + \Const  \kappa^2 \biggl( \beta_3 + \sum_{j \in N(B_i)} \E^{{\mathcal{F}_i}} | \bar X_j/\sigma|^3 + \sum_{j \in N(B_i) } \E^{{\mathcal{F}_i}} \lvert \bar S \bar X_j^3/\sigma^4 \rvert\biggr) .
        \label{eq-aa1}
	\end{equ}
	Rearranging \cref{eq-aa1} yields 
	\begin{equ}
		\MoveEqLeft
		\P^{{\mathcal{F}_i}} \bigl( z + a/\bar V \leq \bar W \leq z + b / \bar V \bigr)\\
		& \leq \frac{2(b - a + \alpha)}{\sigma} + 0.5 L(\alpha)\\
& \quad +   \Const  \kappa^2 \biggl( \beta_3 + \sum_{j \in N(B_i)} \E^{{\mathcal{F}_i}} | \bar X_j/\sigma|^3 + \sum_{j \in N(B_i) } \E^{{\mathcal{F}_i}} \lvert \bar S \bar X_j^3/\sigma^4 \rvert\biggr) .
\label{eq-aa11}
	\end{equ}
	Letting $a = x - 1/n$ and $b = x + 1/n + \alpha$, and $z = y$ in \cref{eq-aa1} and taking supremum over $y, x \in \mathbb{Q}$ and letting $n \to \infty$, we have 
	\begin{align*}
		L(\alpha) & \leq  
		0.5 L(\alpha) +  \Const  \kappa^2 \biggl( \beta_3 + \sum_{j \in N(B_i)} \E^{{\mathcal{F}_i}} | \bar X_j/\sigma|^3 + \sum_{j \in N(B_i) } \E^{{\mathcal{F}_i}} \lvert \bar S \bar X_j^3/\sigma^4 \rvert\biggr)  , 
	\end{align*}
	and hence, 
	\begin{equ}
		L(\alpha) & \leq \Const  \kappa^2 \biggl( \beta_3 + \sum_{j \in N(B_i)} \E^{{\mathcal{F}_i}} | \bar X_j/\sigma|^3 + \sum_{j \in N(B_i) } \E^{{\mathcal{F}_i}} \lvert \bar S \bar X_j^3/\sigma^4 \rvert\biggr) . 
        \label{eq-aa2}
	\end{equ}

	Substituting \cref{eq-aa2} to \cref{eq-aa11},  
	and recalling that $\alpha = 20 \kappa^2 \beta_3$, 
	we have 
	\begin{align*}
		\MoveEqLeft
		\P^{{\mathcal{F}_i}} \bigl( z + a/\bar V \leq \bar W \leq z + b / \bar V \bigr)\\
		& \leq \frac{2(b - a )}{\sigma} +   \Const  \kappa^2 \biggl( \beta_3 + \sum_{j \in N(B_i)} \E^{{\mathcal{F}_i}} | \bar X_j/\sigma|^3 + \sum_{j \in N(B_i) } \E^{{\mathcal{F}_i}} \lvert \bar S \bar X_j^3/\sigma^4 \rvert\biggr) . 
	\end{align*}
This completes the proof. 
\end{proof}

\section{Proof of Main results}%
\label{sub:proof_of_theorem_1_1}

In this subsection, let $\bar X_i, \bar Y_i, \bar S, \bar V$ and $\bar W$ be defined as in \cref{eq-barS}. 
We use a truncation argument to prove \cref{thm1}. Specifically, 
we first prove a Berry--Esseen bound for $\bar W$, and then prove an error bound for $\sup_{z \in \R} \lvert \P(W \leq z) - \P(\bar W \leq z) \rvert$. 
Again, we denote by $C, C_1, C_2, \dots$ absolute positive constants that may take different values in different places. 

Now, we give the following proposition, which provides a Berry--Esseen bound for $\bar W$, and the proof is based on Stein's method and the concentration inequality approach. 
\setcounter{mycounter}{0}
\begin{proposition}
	\label{pro-BEWbar}
	Under \textup{(LD1)} and \textup{(LD2)}. 
	We have 
	\begin{align*}
		\sup_{z \in \R} \bigl\lvert \P(\bar W \leq z) - \Phi(z) \bigr\rvert  \leq 
		C \bigl((\theta +1) \kappa^2 \beta_3 +  \kappa \beta_2\bigr).
	\end{align*}
\end{proposition}
\begin{proof}
Without loss of generality, noting that $\kappa \geq 1$, we assume that $\beta_2 \leq 1/(150\kappa)$ and $\beta_3 \leq 1/(150\kappa^2)$, otherwise the inequality is trivial. 
Let $f_{z}$ be the solution to the Stein equation 
\begin{align*}
	f'(w) - w f(w) = \I(w \leq z) - \P(Z \leq z).
\end{align*}
In what follows, we simply write $f \coloneqq f_z$.
It can be shown that for all $w, w' \in \R$, 
\begin{equ}
	0 \leq f(w) \leq 1, \quad 
	\lvert f'(w) \rvert \leq 1, \quad 
	\lvert f'(w ) - f(w') \rvert \leq 1 ,
    \label{eq-pro1}
\end{equ}
and for all $w$ and $t$, 
\begin{equ}
	\lvert f'(w + t) - f'(w) \rvert
	& \leq ( |w| + 1 ) |t| + 1( z-|t| \leq w \leq z + |t| ).
    \label{eq-pro2}
\end{equ}

Note that by the Stein equation, we have 
\begin{align*}
	\P(\bar W \leq z) - \P(Z \leq z)
	& = \E \{ f'(\bar W) - \bar W f(\bar W) \}.
\end{align*}
Let  $\bar X_i$ and $\bar Y_i$ be defined as in \cref{eq-barS} with $\tau = \kappa$, and recall that 
\begin{equ}
	\xi_i = \frac{\bar X_i}{\bar V}, \quad \eta_i = \frac{\bar Y_i}{ \bar V }	, \quad \bar W^{(i)} =  \bar W - \eta_i.   
    \label{eq-xieta}
\end{equ}
Then, it follows that 
\begin{align*}
	\E \{ \bar W f(\bar W) \} 
	& = \sum_{i \in \mathcal{J}} \E \{ \xi_i (f(\bar W) - f(\bar W^{(i)}) ) \} + \sum_{i \in \mathcal{J}} \E \{ \xi_i f(\bar W^{(i)}) \}\\
	& = \E \biggl\{\frac{1}{\bar V^2}\int_{-\infty}^{\infty} f'\Bigl( \frac{\bar W + u}{\bar V} \Bigr) \hat K(u) du\biggr\} + \sum_{i \in \mathcal{J}} \E \{ \xi_i f(\bar W^{(i)}) \},
\end{align*}
where 
\begin{align*}
	\hat K(u) & =  \sum_{i \in \mathcal{J}} \bar X_i \bigl\{ \I( - \bar Y_i \leq u \leq 0 ) - \I(0 < u \leq -\bar Y_i) \bigr\}.
\end{align*}
By \cref{eq-pro1,eq-pro2}, we then obtain 
\begin{equ}
	\MoveEqLeft\P (\bar W \leq z) - \P(Z \leq z)\\
	& = \E \biggl\{f'(\bar W)\biggl( 1 - \frac{\sum_{i \in \mathcal{J}} \bar X_i \bar Y_i}{\bar V^2}  \biggr)\biggr\} \\
	& \quad - \E \biggl\{ \frac{1}{\bar V^2}\int_{-\infty}^{\infty} \biggl(f'\Bigl( \frac{\bar S + t}{ \bar V } \Bigr) - f'\Bigl( \frac{\bar S}{\bar V} \Bigr)\biggr) \hat M(t) dt \biggr\} - \sum_{i \in \mathcal{J}} \E \{ \xi_i f(\bar W^{(i)}) \}
	\\
	& \leq R_1 + R_2 + R_3 + R_4, 
    \label{eq-RR0}
\end{equ}
where 
\begin{align*}
	R_1 & =  \E \biggl\lvert 1 - \frac{1}{\bar V^2} \sum_{i \in \mathcal{J}} \bar X_i \bar Y_i \biggr\rvert, \\
	R_2 & = \frac{8}{\sigma^3}\E \biggl\{  ( \lvert \bar W \rvert + 1 ) \int_{-\infty}^{\infty} |t \hat M(t)| dt \biggr\},\\
	R_3 & = \frac{4}{\sigma^2} \sum_{i \in \mathcal{J}}\E \Bigl\{ |\bar X_i \bar Y_i|  \I \bigl( z - |\bar Y_i|/\bar V \leq \bar W \leq z + |\bar Y_i|/\bar V \bigr) \Bigr\},\\
	R_4 & = \biggl\lvert \sum_{i \in \mathcal{J}} \E \{ \xi_i f(\bar W^{(i)}) \}  \biggr\rvert. 
\end{align*}

For $R_1$, by \cref{eq-l13,eq-l14,eq-l16}, we have 
\begin{equ}
	R_1 & \leq 
	\E \biggl\{ \biggl\lvert 1 - \frac{1}{2\sigma^2} \sum_{i \in \mathcal{J}} \bar X_i \bar Y_i \biggr\rvert \I\biggl( \sum_{i \in \mathcal{J}} X_i Y_i > 2\sigma^2 \biggr) \biggr\}
	\\
		& \quad + 
	\E \biggl\{ \biggl\lvert 1 - \frac{4}{\sigma^2} \sum_{i \in \mathcal{J}} \bar X_i \bar Y_i \biggr\rvert \I\biggl( \sum_{i \in \mathcal{J}} X_i Y_i	 < \sigma^2/4 \biggr) \biggr\}\\
		& \leq \biggl( 1 + 4 \bar \sigma^2/\sigma^2 \biggr) \P \biggl( \biggl\lvert \sum_{i \in \mathcal{J}} \bar X_i \bar Y_i - \bar\sigma^2 \biggr\rvert > \sigma^2/2 \biggr)\\
		& \quad + \frac{4}{\sigma^2} \E \biggl\lvert \sum_{i \in \mathcal{J}} \bar X_i \bar Y_i - \bar \sigma^2 \biggr\rvert \I \biggl( \biggl\lvert \sum_{i \in \mathcal{J}} \bar X_i \bar Y_i - \bar\sigma^2 \biggr\rvert > \sigma^2/2 \biggr) \\
		& \leq 100 \kappa^2 \beta_3. 
    \label{eq-RR1}
\end{equ}
For $R_2$, we have 
\begin{align*}
	R_2 & \leq \frac{4}{\sigma^3} \sum_{i \in \mathcal{J}}\E \bigl\{ ( \lvert \bar W \rvert + 1 ) |\bar X_i \bar Y_i^2	  | \bigr\}\leq \frac{4}{\sigma^3} \sum_{i \in \mathcal{J}}\E \biggl\{ \Bigl( \frac{2}{\sigma} \lvert \bar S \rvert + 1 \Bigr) |\bar X_i \bar Y_i^2	  | \biggr\}.
\end{align*}
Let 
\begin{align*}
	\bar S^{(B_i)} = \sum_{j \not\in B_i}\bar X_j.
\end{align*}
Then, $\bar S^{(B_i)}$ is independent of $(\bar X_i, \bar Y_i)$, and $\lvert \bar S - \bar S^{(B_i)} \rvert \leq \sigma$. Moreover, by \cref{eq-l16},
\begin{align*}
	\E \lvert \bar S^{(B_i)} \rvert \leq \E \lvert \bar S \rvert + \sigma
	 = \bar \sigma + \sigma \leq 2.02 \sigma. 
\end{align*}
Thus, we have
\begin{align*}
	\E \lvert \bar S \bar X_i \bar Y_i^2 \rvert
	& \leq \E \lvert \bar S^{(B_i)} \rvert \E \lvert \bar X_i \bar Y_i^2 \rvert + \sigma\E \lvert \bar X_i \bar Y_i^3 \rvert \leq 3.02 \sigma \E \lvert \bar X_i \bar Y_i^2 \rvert. 
\end{align*}
Moreover, by H\"older's inequality, for any $c > 0$, 
\begin{align*}
	\sum_{i \in \mathcal{J}} \E \lvert \bar X_i \bar Y_i^2 \rvert 
	& \leq \sum_{i \in \mathcal{J}} \E \biggl( \frac{c^3 \lvert \bar X_i \rvert^3}{3} + \frac{2\lvert \bar Y_i^3 \rvert}{3c^{3/2}} \biggr) \\
	& \leq \sum_{i \in \mathcal{J}} \E \biggl( \frac{c^3 \lvert \bar X_i \rvert^3}{3} + \kappa^2\sum_{j \in A_i} \frac{2\lvert \bar X_j^3 \rvert}{3c^{3/2}} \biggr) \\
	& \leq \frac{c^3 \sigma \beta_3}{3} + \frac{2 \kappa^3 \sigma\beta_3}{3 c^{3/2}}. 
\end{align*}
Choosing $c = \kappa^{2/3}$, we have 
\begin{align*}
	\sum_{i \in \mathcal{J}} \E \lvert \bar X_i \bar Y_i^2 \rvert 
	& \leq \kappa^{2} \sigma^3 \beta_3. 
\end{align*}
Therefore,
\begin{equ}
	R_2 
	& \leq 16 (\kappa + 1) \sum_{i \in \mathcal{J}} \E \lvert \bar X_i \bar Y_i^2/\sigma^3 \rvert\leq 16 \kappa^{2} \beta_3. 
    \label{eq-RR2}
\end{equ}

For $R_3$, by \cref{lem-con2} and noting that $\lvert \bar X_i \rvert \leq \sigma / \kappa$, we have 
\begin{equ}
	R_3 
	& \leq \Const\biggl(  \sum_{i \in \mathcal{J}} \E \lvert \bar X_i \bar Y_i^2 / \sigma^3 \rvert + \kappa^2 \beta_3 \sum_{i \in \mathcal{J}} \E \lvert \bar X_i \bar Y_i / \sigma^2 \rvert \\
	& \quad + \kappa^2 \sum_{i \in \mathcal{J}} \sum_{j \in N(B_i)} \E \lvert \bar X_i \bar Y_i \bar X_j / \sigma^3\rvert + \kappa^2 \sum_{i \in \mathcal{J}} \sum_{j \in N(B_i)} \E \lvert \bar S \bar X_i \bar Y_i \bar X_j / \sigma^4 \rvert \biggr)\\
	& \leq \Const \kappa^2\bigl(   \beta_3 +  \kappa^2 \theta	\beta_3 +  \beta_3 +   \beta_3\bigr) \leq \Const(\theta + 1)\kappa^2 \beta_3. 
    \label{eq-RR3}
\end{equ}

For $R_4$, by \cref{lem2.2}, we have 
\begin{equ}
	R_4 \leq \Const( \kappa^2 \beta_3 +  \kappa\beta_2 ). 
    \label{eq-RR4}
\end{equ}

By \cref{eq-RR0,eq-RR1,eq-RR2,eq-RR3,eq-RR4}, we have 
\begin{align*}
	\P(\bar W \leq z) - \P(Z \leq z)
	& \leq \Const(\theta + 1) \kappa^2 \beta_3 + \Const \kappa \beta_2. 
\end{align*}
The lower bound can be obtained similarly. This completes the proof. 	
\end{proof}

Now we are ready to give the proof of \cref{thm1}. 
\begin{proof}
[Proof of \cref{thm1}]
	Assume without loss of generality that 
	\begin{align*}
		\beta_2 \leq \frac{1}{150\kappa}, \quad \beta_3 \leq \frac{1}{150 \kappa^2}. 
	\end{align*}
	Recall the function $\psi$ in \cref{eq-psi} and $\bar S$ in \cref{eq-barS}, and define 
	\begin{equ}
		\tilde V = \psi\Bigl( \sum_{i \in \mathcal{J}} (\bar X_i \bar Y_i - \bar X \bar Y) \Bigr), \quad \tilde W = \bar S / \tilde V.  
        \label{eq-tilV}
	\end{equ}
	Then, we have 
	\begin{align*}
		\sup_{z \in \R} \bigl\lvert \P(\tilde W \leq z) - \P( W \leq z ) \bigr\rvert
		& \leq \P(\max_{i \in \mathcal{J}} \lvert X_i \rvert > \sigma / \kappa) + \P( \sum_{i \in \mathcal{J}} \bar X_i \bar Y_i \leq \sigma^2/4  ).
	\end{align*}
	Now, 
	\begin{align*}
		\P ( \max_{i \in \mathcal{J}} |X_i| > \sigma/\kappa ) 
		& \leq \sum_{i \in \mathcal{J}} \P(|X_i| > \sigma/\kappa) 
		= \beta_0.
	\end{align*}
	Also, by \cref{eq-l14}, we have 
	\begin{align*}
		\P\biggl\{ \sum_{i \in \mathcal{J}} \bar X_i \bar Y_i \leq \sigma^2/4  \biggr\}
		& \leq \P \bigg\{ \biggl\lvert \sum_{i \in \mathcal{J}} (\bar X_i \bar Y_i - \E \{ \bar X_i \bar Y_i \}) \biggr\rvert \geq \frac{\sigma^2}{2} \biggr\} 
		\leq 4 \kappa^2 \beta_3. 
	\end{align*}
	
	Then, 
	\begin{equ}
		\sup_{z \in \R} \bigl\lvert \P(\tilde W \leq z ) - \P(W \leq z) \bigr\rvert
		 & \leq 4\kappa^2 \beta_3+ \beta_0. 
	    \label{eq-trun}
	\end{equ}

	For $\epsilon > 0$, let 
	\begin{align*}
		h_{z, \epsilon} (w) = 
		\begin{cases}
			1 & \text{if $w \leq z$}, \\
			0 & \text{if $w > z + \epsilon$}, \\
			\text{linear} & \text{otherwise}.
		\end{cases}
	\end{align*}
	Then, it follows that $h_{z,\epsilon}'(w) = \I(z \leq w \leq z + \epsilon)/\epsilon$. 
	Recall that $\bar W$ and $\bar V$ as in $\cref{eq-barS}$, 
	and observe that 
	\begin{equ}
		\P (\tilde W \leq z) - \P(\bar W \leq z)
		& \leq \E \{ h_{z, \epsilon}(\tilde W) - h_{z, \epsilon}(\bar W) \} + \P (z \leq \bar W \leq z + \epsilon) \\
		& \leq \E \{ h_{z, \epsilon}(\tilde W) - h_{z, \epsilon}(\bar W) \} + \bigl( \Phi(z + \epsilon) - \Phi(z)\bigr) \\
		& \quad + 2 \sup_{z \in \R} \lvert \P(\bar W \leq z) - \Phi(z) \rvert. 
	    \label{eq-trun2}
	\end{equ}
	Note that 
	\begin{align*}
		\biggl\lvert \frac{1}{\bar V} - \frac{1}{\tilde V} \biggr\rvert
		& \leq \frac{ |\mathcal{J}| |\bar X \bar Y| }{\bar V \tilde V (\bar V + \tilde V)}, 
	\end{align*}
	and thus
	\begin{equ}
		\lvert \E \{ h_{z, \epsilon}(\bar W) - h_{z, \epsilon}(\tilde W) \} \rvert
		& \leq \|h_{z,\epsilon}'\| \E | \bar W - \tilde W|  \\
		& \leq \frac{ 1 }{\epsilon} \E \biggl\{ \frac{ |\bar S^2 \bar Y| }{ \bar V \tilde V (\bar V + \tilde V) } \biggr\}\\
		& \leq \frac{4}{\epsilon \sigma^3} \E \lvert \bar S^2 \bar Y \rvert\\
		& = \frac{4}{\epsilon \sigma^3 \lvert \mathcal{J} \rvert^{-1}} \E \biggl\lvert \biggl(\sum_{i \in \mathcal{J}} \bar X_i\biggr)^2 \biggl(\sum_{j \in \mathcal{J}} \bar Y_j\biggr) \biggr\rvert.
        \label{eq-p2.1-a}
	\end{equ}
	By the H\"older inequality and \cref{lem-4moment}, we have for any $c > 0$, 
	\begin{align*}
		\E \biggl\lvert \biggl(\sum_{i \in \mathcal{J}} \bar X_i\biggr)^2 \biggl(\sum_{j \in \mathcal{J}} \bar Y_j\biggr) \biggr\rvert 
		& \leq \frac{c}{2} \E \lvert \bar S^2 \rvert + \frac{1}{2c} \E \biggl\lvert \biggl(\sum_{i \in \mathcal{J}} \bar X_i\biggr)^2 \biggl( \sum_{j \in \mathcal{J}} \bar Y_j \biggr)^2\biggr\rvert \\
		& \leq \frac{c}{2} \bar \sigma^2 + \frac{1}{2c} \{1161 \kappa^2(\theta + 1) \sigma^4\}. 
	\end{align*}
	Choosing $c = 35 \kappa (\theta + 1) \sigma$, and noting that $\bar\sigma^2 \leq 1.02 \sigma^2$ by \cref{lem-1}, we have the expectation term of the right hand side of \cref{eq-p2.1-a} is bounded by 
	\begin{equ}
		\E \biggl\lvert \biggl(\sum_{i \in \mathcal{J}} \bar X_i\biggr)^2 \biggl(\sum_{j \in \mathcal{J}} \bar Y_j\biggr) \biggr\rvert
		& \leq 35 \kappa (\theta + 1) \sigma^3. 
        \label{eq-p2.1-b}
	\end{equ}
	Substituting \cref{eq-p2.1-b} to \cref{eq-p2.1-a} yields 
	\begin{align*}
		\lvert \E \{ h_{z, \epsilon}(\bar W) - h_{z, \epsilon}(\tilde W) \} \rvert
		& \leq 140 \kappa(\theta + 1) / (\epsilon |\mathcal{J}|) .  
	\end{align*}
	Moreover, for the second term of the right hand side of \cref{eq-trun2},
	\begin{align*}
		\Phi(z + \epsilon) - \Phi(z) \leq 0.4 \epsilon. 
	\end{align*}
	Choosing $\epsilon = 19 ( \kappa^{1/2}(\theta + 1) )/ \lvert \mathcal{J} \rvert^{1/2}$, and by \cref{eq-trun2}, we have 
	\begin{equ}
		\MoveEqLeft
		\P(\tilde W \leq z) - \P(\bar W \leq z)\\
		& \leq 16 \kappa^{1/2} (\theta + 1) \lvert \mathcal{J} \rvert^{-1/2} + 2  \sup_{z \in \R} \lvert \P(\bar W \leq z) - \Phi(z) \rvert.
	    \label{eq-trun3}
	\end{equ}
	The same lower bound also holds by the same argument. 
	Then, we have 
	\begin{equ}
		\MoveEqLeft
		\sup_{z \in \R} |\P(\tilde W \leq z) - \P(\bar W \leq z)| \\
		& \leq 
		16  \kappa^{1/2} (\theta + 1) \lvert \mathcal{J} \rvert^{-1/2} + 2 \sup_{z \in \R} \lvert \P(\bar W \leq z) - \Phi(z) \rvert.
        \label{eq-trun4}
	\end{equ}
	By \cref{eq-trun,eq-trun4}, and applying \cref{pro-BEWbar}, we completes the proof. 
\end{proof}
\section*{Acknowledgements}
The author would like to thank Qi-Man Shao for his helpful discussions. This project was supported by the Singapore Ministry of Education Academic Research Fund Tier 2 grant MOE2018-T2-2-076.

\end{document}